\documentclass[a4paper,10pt,reqno]{amsart}
\usepackage{amsmath, amsthm, amsfonts, amssymb}
\usepackage{color}
\usepackage{graphicx}
\usepackage{soul}
\usepackage{enumitem}

\usepackage[dvipsnames]{xcolor}
\usepackage{framed}
\definecolor{shadecolor}{RGB}{180,180,180}

\usepackage{tikz}
\usetikzlibrary{arrows,automata,backgrounds,calendar,mindmap,trees}
\usepackage{xcolor}

\usepackage[colorlinks=true, linkcolor = black, citecolor = black]{hyperref}
\usepackage{hyperref}

\usepackage[margin=1in]{geometry}

\newcommand{\C}{\mathcal{C}}
\renewcommand{\S}{\mathcal{S}}
\newcommand{\R}{\mathcal{R}}

\renewcommand{\aa}{\boldsymbol{a}}

\newcommand{\Ainf}{A_{\infty}}
\newcommand{\Binf}{B_{\infty}}
\newcommand{\Cinf}{C_{\infty}}

\newcommand{\cc}{\mathbf{c}}
\newcommand{\RR}{\mathbf{R}}
\newcommand{\ww}{\mathbf{c}_{\infty}}

\newcommand{\CC}{\mathbf{C}}
\newcommand{\WW}{\mathbf{C}_{\infty}}
\def\pa{\partial}
\def\mathbf{\boldsymbol}

\begin{document}

\title[Complex balanced chemical reaction networks]
{Trend to equilibrium for reaction-diffusion systems arising from complex balanced chemical reaction networks
}

\author[L. Desvillettes, K. Fellner, B.Q. Tang] {Laurent Desvillettes, Klemens Fellner and Bao Quoc Tang}

\address{Laurent Desvillettes\hfill\break
	Univ. Paris Diderot, Sorbonne Paris Cit\'{e}, Institut de Math\'{e}matiques de Jussieu - Paris Rive
	Gauche, UMR 7586, CNRS, Sorbonne Universit\'{e}s, UPMC Univ. Paris 06, F-75013, Paris, France.}
\email{desvillettes@math.univ-paris-diderot.fr}

\address{Klemens Fellner\hfill\break
	Institute of Mathematics and Scientific Computing, University of Graz, Heinrichstrasse 36, 8010 Graz, Austria.}
\email{klemens.fellner@uni-graz.at}

\address{Bao Quoc Tang$^{1, 2}$ \hfill\break
$^{1}$ Institute of Mathematics and Scientific Computing, University of Graz, Heinrichstrasse 36, 8010 Graz, Austria\hfill\break
$^{2}$ Faculty of Applied Mathematics and Informatics, Hanoi University of Sience and Technology, 1 Dai Co Viet, Hai Ba Trung, Hanoi, Vietnam}
\email{quoc.tang@uni-graz.at} 


\begin{abstract}
	The quantitative convergence to equilibrium for reaction-diffusion systems arising from complex balanced chemical reaction networks with mass action
	kinetics is studied by using the so-called entropy method. 
		In the first part of the paper, by deriving explicitly the entropy dissipation, we show that for complex balanced systems without boundary equilibria, each trajectory converges exponentially fast to the unique complex balance equilibrium. Moreover, a constructive proof is proposed to explicitly estimate the rate of convergence in the special case of a cyclic reaction. 
	In the second part of the paper, complex balanced systems with boundary equilibria are considered. We
	focus on a specific case involving 
	 three chemical substances
	 for which the boundary equilibrium is shown to be unstable in some sense, so that  exponential convergence to the unique strictly positive equilibrium is recovered.
\end{abstract}

\maketitle
\numberwithin{equation}{section}
\newtheorem{theorem}{Theorem}[section]
\newtheorem{lemma}[theorem]{Lemma}
\newtheorem{proposition}[theorem]{Proposition}
\newtheorem{definition}{Definition}[section]
\newtheorem{remark}{Remark}[section]
\newtheorem{corollary}[theorem]{Corollary}
\newtheorem*{gac*}{Global Attractor Conjecture}

\tableofcontents
\section{Introduction and Main results}
%
%
%
%
%
%

\medskip
The foundation of the study of chemical reaction networks goes back to the pioneering works of F.J.M. Horn, R. Jackson, A.I. Volpert and M. Feinberg,  (see e.g. \cite{HoJa72, Vol72, FeHo74, Fe79, Fe87}). The aim of the theory is to study the behaviour of chemical reaction networks regardless of specific values of the reaction rate constants, 
since these rates may be hard to determine in some practical situations. One of the main questions of the theory is the large time asymptotic behaviour of the dynamical system corresponding to the reaction network. While the ODE setting of chemical reaction networks has been extensively studied in the literature, the PDE setting is much less investigated. This is the main motivation of this work. The ultimate aim of this paper is to investigate the large time asymptotic behaviour of reaction-diffusion systems arising from chemical reaction networks. More precisely, by exploiting the so-called entropy method, we look for  {\it quantitative estimates of convergence to equilibrium} for a class of chemical reaction networks called {\it complex balanced reaction networks}.


\medskip
%

Before providing a general notation of chemical reaction networks, we shall consider the example of a single reaction $S_1 + S_2 \xrightarrow{k} S_3$ where $k>0$ is a reaction rate constant. To describe this reaction, we introduce the set of {\it chemical substances} $\mathcal S = \{S_1, S_2, S_3\}$, the set of {\it chemical complexes} $\mathcal C = \{(1,1,0); (0,0,1)\}$ corresponding to $S_1 + S_2$ and $S_3$ and the set of {\it chemical reactions} $\mathcal{R} = \{(1,1,0) \xrightarrow{k} (0,0,1) \}$. Now, by following the notation of e.g. \cite{HoJa72}, a general chemical reaction network consists of $N$ {chemical substances} $\mathcal S = \{S_1, \ldots, S_N\}$, a finite set $\mathcal C$ of {chemical complexes}, which appear on either sides of a chemical reaction, i.e. 
$\mathcal C = \{y\}\subset {{\mathbb{N}^N}}$
(This is a slight abuse of notation. We should indeed write $\{y^{(1)},..,y^{(|C|)} \}$, but we do not because superscripts with another meaning will be used in the sequel)
 and the set of {chemical reactions} $\mathcal R = \{y_r \xrightarrow{k_r} y_r': y_r, y_r' \in \mathcal C,\, k_r>0\}$. We shall only consider chemical reaction networks which satisfy the following natural assumptions:  
\begin{definition}[Chemical reaction networks]\label{def:crn}\hfill\\
	Let $\S = \{S_i\}_{i=1}^{N}$, $\C = \{y\}\subset \mathbb{N}^N$ and $\R = \{y_r \xrightarrow{k_r} y_r'\}_{r=1,\ldots,|\R|}$ denote finite sets of species, complexes and reactions, respectively. The triple $\{\S, \C, \R\}$ is called a chemical reaction network as long as the following three natural requirements are met:
	\begin{itemize}
		\item[1.] For each $S_i \in \S$, there exists at least one complex $y\in \C$ in which the stoichiometric coefficient $y_i$ of $S_i$ is strictly positive (subscripts for $y$ are used both to prescribe stoichiometric coefficients (usually $i$ is then used) and to prescribe the index of a chemical reaction (usually $r$ is then used); when both need to be prescribed, we use $y_{r,i}$);
		\item[2.] There is no trivial reaction $y\rightarrow y\in \R$ for any complex $y\in\C$;
		\item[3.] For any $y\in \C$, there must exist $y'\in \C$ such that either $y\rightarrow y'\in \R$ or $y' \rightarrow y\in \R$.
	\end{itemize}
\end{definition}

 We assume that the chemical species corresponding to the reaction network $\{\mathcal S, \mathcal C, \mathcal R\}$ are contained in a bounded vessel (or reactor) $\Omega\subset \mathbb R^n$, where $\Omega$ is a bounded smooth ($C^2$) domain. 
  Moreover, we denote by $\cc(x,t) = (c_1(x,t), \ldots, c_N(x,t))$ the vector of concentrations, where $c_i(x,t)$ is the concentration of the specie $S_i$ at time $t>0$ and position $x\in\Omega$. Each substance $S_i$ is assumed to diffuse in $\Omega$ with a strictly positive diffusion coefficient $d_i>0$. 
 The corresponding reaction-diffusion system reads as
\begin{equation}\label{e0}
\frac{\partial}{\partial t}\cc - \mathbb D\Delta \cc = \mathbf{R}(\cc) \quad \text{ for } \quad (x,t)\in \Omega\times \mathbb R_+ ,
\end{equation}
where the diffusion matrix $\mathbb D = \mathrm{diag}(d_1, \ldots, d_N)$ is 
positive definite (since $d_i>0$ for all $i=1,\ldots,N$). Moreover,  by applying the {\it law of mass action kinetics}, 
the reaction vector $\mathbf{R}(\cc)$ is modelled as
\begin{equation}\label{Reaction}
\mathbf{R}(\cc) = \sum_{r=1}^{|\mathcal R|}k_r\cc^{y_r}(y_r' - y_r), \quad \text{ with } \quad \cc^{y_r} := \prod_{i=1}^{N}c_i^{y_{r,i}},
\end{equation}
where $k_r>0$ denotes the rate constant of the $r$-th reaction. Finally, system \eqref{e0} is associated to
nonnegative initial data ${\mathbf{c}}_0(x) \ge 0$ (by which we mean ${\mathbf{c}}_0(x) := (c_{1,0}(x),..,c_{N,0}(x))$ and $c_{i,0}(x) \ge 0$ for $i=1,..,N$), and homogeneous Neumann boundary condition 
\begin{equation}\label{Nin}
 \cc(x,0) = {\mathbf{c}}_0(x) \,\, {\hbox{ for }}\,\,  x\in \Omega ,
\quad\text{ and } \quad 
\nabla \cc \cdot \nu = 0 \,\,{\hbox{ for }}\,\, (x,t)\in \partial\Omega \times \mathbb{R}_+ ,
 \end{equation} 
where $\nu := \nu(x)$ is the outward normal unit vector at point $x \in \pa\Omega$.
 \medskip 
 
Concerning the set of chemical reactions $\mathcal R$, we denote by {$W$ the Wegscheider matrix, i.e.
	\begin{equation*} 
	{W} = [(y_r' - y_r)_{r=1,\ldots,{|\R|}}]^{\top} \in \mathbb{R}^{|\R|\times N},
	\end{equation*}
and $m = \mathrm{dim}(\mathrm{ker}(W))$}. Then, if $m>0$, there exists a (non-unique) matrix $\mathbb Q\in \mathbb R^{m\times N}$  such that 
\begin{equation}\label{11ter}
	\mathrm{rank}(\mathbb Q) = m \qquad \text{ and } \qquad \mathbb Q\, \mathbf{R}(\cc) = 0 \; \text{ for all } \; \cc\in \mathbb R^N_{\geq 0} \,(:= (\mathbb R_+)^N),
\end{equation}
so that the rows of $\mathbb Q$ form a basis of {$\mathrm{ker}(W)$}. 
Thus, because of the homogeneous Neumann boundary condition, the use of the matrix $\mathbb Q$ leads, at the formal level,  to the following {\it mass conservation laws} for the solutions of eq. (\ref{e0}):
\begin{equation*}
	\mathbb Q\,\overline{\cc}(t) = \mathbb Q\,\overline{\cc_0} =: M/|\Omega| \in \mathbb R^m \quad \text{ for all } \quad t>0,
\end{equation*}
where $\overline{\cc} := (\overline{c_1}, \ldots, \overline{c_N})$ and $\overline{c_i} := \int_{\Omega}c_i \frac{dx}{|\Omega|}$, and $M$ is called the {\it initial mass vector}. By changing the signs of some rows of $\mathbb Q$ if necessary, we can always consider (w.l.o.g.) a matrix $\mathbb Q$ such that the initial mass vector $M$ is nonnegative, i.e. $M \in \mathbb R_{\ge0}^m$.
\medskip 

To state our results, we require the following definitions concerning {\it equilibria} of chemical reaction networks.

\begin{definition}[Equilibria]\label{Equilibria} \hfill\\
Consider a chemical reaction network with a reaction vector $\mathbf{R}$
defined by (\ref{Reaction})
  and  a mass vector $M \in \mathbb R^m_{\geq 0}$.
	Let $\ww := (c_{1,\infty},..,c_{N,\infty})\in \mathbb{R}^N_{\geq 0}$ be such that $\mathbb Q\,\ww = M$. Then,
	\begin{itemize}
		\item $\ww$ is called an {\normalfont{equilibrium}}  if $\mathbf{R}(\ww) = 0$.
		
		\item $\ww$ is called a {\normalfont{detailed balance equilibrium}}  if for each forward reaction $y \xrightarrow{k_{f}} y'$ (with $k_f>0$) in $\mathcal R$, there exists in $\mathcal{R}$ also the corresponding backward reaction $y' \xrightarrow{k_{b}} y$ (with $k_b>0$) and 
\begin{equation*}
		k_f\ww^{y} = k_b\ww^{y'}.
\end{equation*}
		
\item 	$\ww$ is called a {\normalfont{complex balance equilibrium}} if for any complex $y\in \C$, we have
\begin{equation} \label{ComplexBalance}
\sum_{\{r:\, y_r = y\}}k_r\ww^{y_r} = \sum_{\{s:\, y_s' = y\}}k_s\ww^{y_s}.
\end{equation}
\par 	 
Roughly speaking, a state $\ww$ is a complex balance equilibrium if and only if the total out- and inflow at $\ww$ are equal for every complex $y$. 
		
\item $\ww$ is called a {\normalfont{boundary detailed/complex balance equilibrium}} (or shortly a {\normalfont{boundary equilibrium}}) if $\ww$ is a detailed/complex balance equilibrium and $\ww\in \partial\mathbb{R}^N_{\geq 0}$.
	\end{itemize}		

A chemical reaction network is called  {\it complex balanced} if it possesses a strictly positive (i.-e., not a boundary) complex balance equilibrium for each 
strictly positive mass vector $M \in \mathbb R^m_{>0}$.
\end{definition}
\medskip

It follows directly from the above definitions that
\begin{equation*}
\ww \text{ is a detailed balance equilibrium} \Rightarrow \ww \text{ is a complex balance equilibrium} \Rightarrow \ww \text{ is an equilibrium},
\end{equation*}
but the reverse is in general not true. 
\medskip

The principle of detailed balance goes back as far as Boltzmann for modelling collisions in kinetic gas theory and for proving the H-theorem for Boltzmann's equation \cite{Bol1896}. It was then applied to chemical kinetics by Wegscheider \cite{Weg1901}. The complex balance condition was also considered by Boltzmann \cite{Bol1887} under the name \emph{semi-detailed balance condition} or \emph{cyclic balance condition}, and was 
systematically used by Horn, Jackson and Feinberg in the seventies, see e.g. \cite{Hor72, FeHo74}.

It is by now a well-known fact for complex balanced chemical reaction networks that all equilibria (of such systems) are complex balanced (cf. \cite{Hor72}). 
Moreover, for each strictly positive initial mass vector ${M \in \mathbb R^m_{> 0}}$, there exists (for such systems) a unique strictly positive (that is, $\ww \in \mathbb{R}^N_{> 0}$) complex balance equilibrium satisfying the mass conservation laws. One or several boundary equilibria (for such systems) can nevertheless exist (cf. \cite{HoJa72}). 

There is also an extensive literature concerning the large time asymptotics of complex balanced systems in the ODE setting, 
i.e. by considering $\cc \equiv \cc(t)$, which satisfies the ODE system
\begin{equation}\label{e00}
	\frac{d}{d t}\cc = \mathbf{R}(\cc),
	\end{equation}
where $\mathbf R(\cc)$ is defined as \eqref{Reaction}. Indeed, it is proven that the unique strictly positive complex balance equilibrium of an ODE reaction network (of such a system) is locally stable (cf. \cite{HoJa72}). 
Moreover, it is conjectured that this equilibrium is in fact globally stable, i.e. that it is the unique global attractor for the dynamical system given by the ODE network (with exception of the boundary equilibria points).
This statement is usually called the {\it Global Attractor Conjecture} and has remained one of the most important open problems in the theory of chemical reaction networks, see e.g. \cite{And11, CNN13, GMS14, Pan12} and the references therein.  
A recently proposed proof of this conjecture in the ODE setting is currently under verification \cite{Cra15}.

For the rest of the paper, we shall systematically consider only complex balance systems. We shall refer to the unique strictly positive complex balance equilibrium (of our systems) as ``the'' (strictly positive) complex balance equilibrium, while all the other complex balance equilibria (which are necessarily boundary equilibria)
are simply named boundary equilibria. 
\medskip

{\bf The first part of this paper} is devoted to the quantitative study of the convergence to equilibrium for complex balanced systems \eqref{e0} of reaction diffusion PDEs without boundary equilibria, and is based on the
so-called {\it entropy method}. 

The  idea of the entropy method consists in studying the large-time asymptotics of a dissipative PDE model by looking for a nonnegative (convex) Lyapunov functional $\mathcal{E}(f)$ and its nonnegative dissipation
\begin{equation} \label{defd}
	\mathcal{D}(f(t)) := -\frac{d}{dt}\mathcal{E}(f(t)),
\end{equation}
along the flow of the PDE model. We shall consider entropy dissipation functionals which are well-behaved in the sense that  firstly, all states satisfying $\mathcal{D}(f) = 0$ (and coherent with the conservation laws of the PDE) correspond to a unique (entropy-minimising) equilibrium $f_{\infty}$, i.e.
\begin{equation*}
	\mathcal{D}(f) = 0 \quad \text{ and } \quad \text{ coherence with conservation laws } \qquad \Longleftrightarrow \qquad f = f_{\infty},
\end{equation*}
and secondly, there exists an {\it entropy entropy-dissipation estimate} of the form
\begin{equation}\label{eede}
	\mathcal{D}(f)\geq \Phi(\mathcal{E}(f) - \mathcal{E}(f_{\infty})), 
\end{equation}
which holds for all $f$ coherent with the conservation laws of the PDE, and
for some nonnegative function $\Phi$ satisfying $\Phi(x) = 0 \Leftrightarrow x = 0$. 
 In this method, if $\Phi'(0)\not= 0$, one gets (at least formally) exponential convergence toward $f_{\infty}$ in relative entropy $\mathcal{E}(f) - \mathcal{E}(f_{\infty})$, with a rate given by some variant of Gronwall's lemma. 
\medskip

It is well known in the ODE theory of complex balanced reaction systems \eqref{e00}
that the free energy
\begin{equation}\label{FreeEnergy}
	E(\cc|\ww) = \sum_{i=1}^{N}\left(c_i\log\frac{c_i}{c_{i,\infty}} - c_i + c_{i,\infty}\right),\qquad\qquad\qquad\quad \text{(ODE entropy functional)}
\end{equation}
is decreasing in time along the solutions of the ODE, i.e. $\frac{d}{dt}E(\cc(t)|\ww) \leq 0$ when $\cc : = \cc(t)$ satisfies eq. \eqref{e00} (cf. \cite{HoJa72, Fe79, Gop13, And14}), and $\ww$ is the complex balance equilibrium of the network. This suggests to consider as an entropy functional for the PDE system \eqref{e0} the quantity
\begin{equation}\label{FreeEnergy_PDE}
\qquad\quad\mathcal{E}(\cc|\ww) = \sum_{i=1}^{N}\int_{\Omega}\left(c_i\log\frac{c_i}{c_{i,\infty}} - c_i + c_{i,\infty}\right)dx, \qquad\qquad\quad \text{(PDE entropy functional)}.
\end{equation}
It is remarkable, however, that the monotonicity of $ t \mapsto {E}(\cc|\ww)$ is in general
 shown indirectly, so that the explicit form of the entropy dissipation $ D(\cc):=-\frac{d}{dt}{E}(\cc|\ww)\geq 0$ is not written down.
  In this paper, we  therefore compute the explicit form of the entropy dissipation functionals ${D}(\cc) = -dE/dt$ and $\mathcal{D}(\cc) = -d\mathcal E/dt$, which is necessary for applying the entropy-entropy dissipation method. More precisely,
we show that
\begin{equation}\label{ED_PDE}
	\mathcal{D}(\cc) := \sum_{i=1}^{N}d_i \int_{\Omega}\frac{|\nabla c_i|^2}{c_i}dx + \sum_{r=1}^{|\mathcal R|}k_r\ww^{y_r}\int_{\Omega}\Psi\left(\frac{\cc^{y_r}}{\ww^{y_r}}; \frac{\cc^{y_r'}}{\ww^{y_r'}}\right)dx 
\end{equation}
satisfies  \eqref{defd}, where $\mathcal{E}$ is defined by \eqref{FreeEnergy_PDE}, when ${\mathbf {c}}$ is a (sufficiently integrable) solution of eq. \eqref{e0}.
In \eqref{ED_PDE}, $\Psi: (0,+\infty)\times (0,+\infty) \rightarrow \mathbb R_+$ is defined by
\begin{equation}\label{psi}
\Psi(x; y) := x\log(x/y) - x + y   \quad  (\ge 0), 
\end{equation}
 and it is easy to check that $\Psi(x; y)=0 \iff x=y$.
We refer to Prop. \ref{cale} below for the verification that  
$\mathcal{D}$ indeed satisfies   \eqref{defd}. 
\medskip

Our first main result concerning complex balanced systems of reaction-diffusion PDEs without boundary equilibria is stated in Theorem \ref{EEDE} below. 
We show that all renormalised solutions, which satisfy a weak entropy entropy-dissipation law, converge  exponentially fast to the unique strictly positive complex balance equilibrium. 

\begin{theorem}[Exponential convergence to equilibrium for complex balanced systems without boundary equilibria]\label{EEDE}\hfill\\
Let $\Omega$ be a bounded smooth $(C^2)$ domain of $ \mathbb{R}^n$. Consider a (positive definite) diffusion matrix $\mathbb D = \mathrm{diag}(d_1, \ldots, d_N)$ with $d_i>0$ for $i=1,2,\ldots, N$. Let
$\mathbf R$ be a  reaction term describing a complex balanced chemical reaction network
without boundary equilibria (coming out of the mass action law, cf. \eqref{Reaction}), and fix a strictly positive initial mass vector $M\in \mathbb{R}^m_{>0}$.
 We recall that such a network possesses a unique strictly positive equilibrium $\ww$. 
	\par 
Then, there exists a constant $\lambda>0$ 
such that for all nonnegative measurable
functions $\cc = (c_1, \ldots, c_N):\Omega \rightarrow \mathbb{R}_+^N$ satisfying the mass conservation laws $\mathbb{Q}\,\overline{\cc} = M$ (cf. \eqref{11ter} for the definition of $\mathbb{Q}$) and the upper bound $\mathcal  E (  \bar{\cc} | \ww) \le K$,
the following entropy-entropy dissipation inequality holds:
	\begin{equation}\label{main}
	\mathcal{D}(\cc) \geq \lambda\, \mathcal{E}(\cc|\ww),
	\end{equation}
	where $\mathcal{E}(\cc|\ww)$ and $\mathcal{D}(\cc)$ are defined in \eqref{FreeEnergy_PDE} and \eqref{ED_PDE}, respectively, and $\lambda$ solely depends on $\Omega$, $d_i$, $M$, 
 the coefficients appearing in $ \mathbf{R}$, and an upper bound $K$ on the relative entropy 
 $\mathcal E(\overline{\cc}|\ww)$
of $\bar{\cc}$ w.r.t. $\ww$.
	\par 
Moreover, following \cite{Fis15}, for any nonnegative initial data with finite initial mass and relative entropy, i.e. $\cc_0\in (L^1(\Omega))^N$ and $\mathcal{E}(\cc_0|\ww) < +\infty$, the system \eqref{e0}--\eqref{Nin} possesses a global renormalised solution $\cc(x,t) = (c_1(x,t), \ldots, c_N(x,t))$ (cf. Rmk. \ref{rmk11} for a precise definition of this solution concept).
\par 
Then, each such renormalised solution which moreover satisfies 
the weak entropy entropy-dissipation law 
\begin{equation}\label{integrability}
	\mathcal{E}(\mathbf c|\mathbf c_{\infty})(t) + \int_{s}^{t}\mathcal{D}(\mathbf c)(r)dr \leq \mathcal E(\mathbf c|\mathbf c_{\infty})(s),\qquad a.a. \quad t\geq s >0, 
\end{equation}
converges in $L^1$-norm
exponentially fast to the (unique, strictly positive) complex balance equilibrium. That is, for a constant $C>0$ depending on the same parameters as $\lambda$ and on $\cc_0$, there holds
\begin{equation}\label{expconverge}
\sum_{i=1}^{N}\|c_i(t) - c_{i,\infty}\|_{L^1(\Omega)}^2 \leq C\,e^{-\lambda t}, \qquad \forall t>0.
\end{equation}
\end{theorem}


\begin{remark}\label{rmk11}
The statement of Thm. \ref{EEDE} is formulated for renormalised solutions, which is the only available concept of global solutions for general systems \eqref{e0}--\eqref{Nin}.
We recall here that a vector of measurable nonnegative components $\cc(x,t) = (c_1(x,t), \ldots, c_N(x,t))$ is said to be a renormalised solution when $c_i \in L^{\infty}_{loc}(\mathbb R_+; L^1(\Omega))$, $\sqrt{c_i} \in L^2_{loc}(\mathbb R_+; H^1(\Omega))$,  and when
 for any smooth function $\xi: \mathbb R_+^N \rightarrow \mathbb R$ with compactly supported derivative $\nabla \xi$ and every $\psi \in C^{\infty}(\overline{\Omega}\times \mathbb R_+)$, the equation 
\begin{equation}\label{renormalised}
	\begin{aligned}
		\int_{\Omega}\xi(\cc(\cdot, T))\psi(\cdot, T)&\,dx - \int_{\Omega}\xi(\cc_0)\psi(\cdot, 0)\,dx - \int_{0}^{T}\!\!\int_{\Omega}\xi(\cc)\frac{d}{dt}\psi\, dx dt\\
		= &-\sum_{i,j=1}^{N}\int_{0}^{T}\!\!\int_{\Omega}\psi \,\partial_{i}\partial_{j}\xi(\cc)(d_i\nabla c_i)\!\cdot\! \nabla c_j\, dx dt\\
		& - \sum_{i=1}^{N}\int_{0}^{T}\!\!\int_{\Omega}\partial_{i}\xi(\cc)(d_i\nabla c_i)\!\cdot\! \nabla \psi\, dx dt
		 + \sum_{i=1}^{N}\int_{0}^{T}\!\!\int_{\Omega}\partial_i\xi(\cc)R_i(\cc)\psi\, dx dt
	\end{aligned}
\end{equation}
holds for almost every $T>0$, with $\mathbf R(\cc) = (R_1(\cc), \ldots, R_N(\cc))$.
\par 
We notice that the assumed weak entropy
entropy-dissipation law \eqref{integrability}, which holds at the formal level, is in fact not easy to prove
in general for 
renormalised solutions, because of the lacking integrability of the 
reaction terms in the entropy dissipation functional \eqref{ED_PDE}. In  
 Thm. \ref{c_lem:eede} below, exponential equilibration of renormalised solutions (without any extra assumptions) in the case of a cyclic reaction network is obtained thanks to estimates which are uniform w.r.t. an approximation process.


For classical and sufficiently integrable weak solutions of \eqref{e0}--\eqref{Nin}, however, the entropy entropy-dissipation law \eqref{integrability} can be verified rigorously (even with an equality sign).  
This was done, for instance, in \cite{DFPV} for weak $(L\, log L)^2$-solutions of a system with quadratic nonlinearities. 
More recently, sufficiently integrable weak solutions were shown to exist for some special systems of the form \eqref{e0}--\eqref{Nin} when a (space dimension-dependent)  ``closeness'' assumption on the diffusion coefficients of the system is made, see e.g. 
\cite{CDF14,FL15}. Imposing an even stronger  ``closeness'' assumption on the diffusion coefficients allows in fact to show the existence of classical solutions, see e.g. \cite{FLS16}.
\end{remark}


\begin{remark} \label{r11}
In the above theorem, the constant $C$ in \eqref{expconverge} can be explicitly estimated.
 In fact, one can take $C = C_{CKP}^{-1}\, \mathcal{E}(\cc_0|\ww)$, where $C_{CKP}$ is the constant in a Csisz\'{a}r-Kullback-Pinsker inequality (cf. Lemma \ref{CKPinequality} below).
\par 
It is however not possible with our method used in Thm. \ref{EEDE} to estimate $\lambda$ in a completely explicit way. 
It is nevertheless possible to define the constant $\lambda$ by a finite-dimensional minimisation problem  (cf. the proof of Thm. \ref{EEDE}). In particular, one needn't use any    
 abstract infinite-dimensional compactness argument.
\end{remark}

The proof of Thm. \ref{EEDE} exploits an additivity property of the entropy functional, the Logarithmic Sobolev inequality (allowing to avoid the use of $L^{\infty}$-bounds, which are unavailable for general systems) and then uses a convexification argument presented in \cite{MiHaMa14}, which allows to further reduce the functional inequality \eqref{eede} into a finite dimensional inequality. The proof of {this} finite-dimensional inequality is based on Taylor expansions around the equilibrium (unfortunately
not yielding an explicit bound for $\lambda$, since the actual value of $\lambda>0$ might be obtained at states far from equilibrium).
\medskip 

The second main result on complex balanced systems without boundary equilibria is stated in Theorem \ref{c_lem:eede} below. 
Since the method of convexification used in Theorem \ref{EEDE} does not yield  explicit 
estimates for the convergence rate $\lambda$ (at least for general systems), we also propose a constructive method, inspired by \cite{DeFe08, FL15, FT15}, 
to prove an entropy entropy-dissipation estimate with computable constants. Another advantage of this method is that it can be applied uniformly w.r.t. the approximating systems used for constructing renormalised solutions, and thus yields 
exponential convergence to equilibrium of renormalised solutions without having to assume the 
weak entropy entropy-dissipation law \eqref{integrability} as in Thm. \ref{EEDE}.

The proposed method consists of four steps. The first three are proven for general complex balanced systems, but  
the last one, which crucially depends on the structure of the conservation laws, is rather a {\it proof of concept} which can be detailed only once a specific system is given. We demonstrate the details of 
this last step 
for the specific case of cycles of reactions connecting an arbitrary number of chemical substances:
\begin{equation*}
	\alpha_1\mathcal{A}_1 \xrightarrow{k_1} \alpha_2\mathcal{A}_2 \xrightarrow{k_2} \ldots \xrightarrow{k_{N-1}} \alpha_{N}\mathcal{A}_N \xrightarrow{k_N} \alpha_1\mathcal{A}_1 ,
\end{equation*}
where $\alpha_i \in {\mathbb{N}_{>0}}$, and $k_i>0$ for $i=1,2,\ldots,N$. 	The corresponding reaction-diffusion system writes as
\begin{equation}\label{c_cycle_0}
\begin{cases}
\partial_ta_1 - d_1\Delta a_1 = \alpha_1(-k_1a_1^{\alpha_1} + k_Na_N^{\alpha_N}), &x\in\Omega, \quad t>0,\\
\partial_ta_2 - d_2\Delta a_2 = \alpha_2(-k_2a_2^{\alpha_2} + k_1a_1^{\alpha_1}), &x\in\Omega, \quad t>0,\\
\cdots\\
\partial_ta_N - d_N\Delta a_N = \alpha_N(-k_Na_N^{\alpha_N} + k_{N-1}a_{N-1}^{\alpha_{N-1}}),\quad &x\in\Omega, \quad t>0,
\end{cases}
\end{equation}	
together with homogeneous Neumann boundary conditions and initial data \eqref{Nin}
(with $\cc$ replaced by ${\mathbf{a}} := (a_1,..,a_N)$). 
The system \eqref{c_cycle_0} has one mass conservation law corresponding to $\mathbb{Q}=\left(\frac{1}{\alpha_1}, \ldots, \frac{1}{\alpha_N}\right)$, i.e.
\begin{equation}\label{MassCons}
\sum_{i=1}^{N}\frac{1}{\alpha_i}\int_{\Omega}a_i(x,t)dx = M:= \sum_{i=1}^{N}\frac{1}{\alpha_i}\int_{\Omega}a_{i,0}(x)dx>0 \quad \text{ for all } \quad t>0,
\end{equation}
where the initial total mass $M$ is assumed to be strictly positive. As a consequence, the system \eqref{c_cycle_0} has a unique strictly positive complex balance equilibrium
 ${\mathbf{a}}_{\infty} = (a_{1,\infty}, \ldots, a_{N,\infty})$ (cf. Lemma \ref{c_equilibrium}). 
\par 
 By denoting ${\mathbf{a}} = (a_1, a_2, \ldots, a_N)$, ${\mathbf{a}}_0 = (a_{1,0}, \ldots, a_{N,0})$
 and by using the periodic notation $a_{N+1} := a_1$, $a_{N+1,\infty} := a_{1,\infty}$ and $a_{N+1,0} := a_{1,0}$, 
 the quantities \eqref{FreeEnergy_PDE} and \eqref{ED_PDE} write in this specific case as
\begin{equation}\label{ee1}
	\mathcal{E}({\mathbf{a}}|{\mathbf{a}}_{\infty}) = \sum_{i=1}^{N}\int_{\Omega}\left(a_i\log{\frac{a_i}{a_{i,\infty}}} - a_i + a_{i,\infty}\right)dx ,
\end{equation}
and
\begin{equation}\label{ee2}
	\mathcal{D}({\mathbf{a}}) = \sum_{i=1}^{N}d_i\int_{\Omega}\frac{|\nabla a_i|^2}{a_i}dx + \sum_{i=1}^{N}k_ia_{i,\infty}^{\alpha_i}\int_{\Omega}\Psi\left(\frac{a_i^{\alpha_i}}{a_{i,\infty}^{\alpha_i}}; \frac{a_{i+1}^{\alpha_{i+1}}}{a_{{i+1},\infty}^{\alpha_{i+1}}} \right)dx,
\end{equation}
where we recall that $\Psi$ is defined by \eqref{psi}.
\medskip

We now state the 
\medskip

	\begin{theorem}[Explicit convergence to equilibrium for a cyclic reaction]\label{c_lem:eede}\hfill\\
Let $\Omega$ be a bounded smooth ($C^2$) domain of ${\mathbb R}^n$, 
 $d_i >0$,
 $\alpha_i \in  {\mathbb{N}_{>0}}$, and $k_i>0$ for all $i=1,2,\ldots,N$.
 We finally fix a strictly positive initial mass $M>0$. 	
 	
Then, there exists a constant $\lambda>0$ which can be {\normalfont explicitly estimated} in terms of 
$\Omega$,
 $d_i$, $M$, 
 $\alpha_i$ and $k_i$,
    such that for any nonnegative measurable functions $a_i:\Omega\rightarrow \mathbb{R}_+$, ($i=1,2,\ldots, N$) satisfying the mass conservation law $\sum_{i=1}^N\frac{1}{\alpha_i} \int_{\Omega} a_i(x)\, dx = M$, we have
\begin{equation*}
\mathcal{D}({\mathbf{a}}) \geq \lambda\, \mathcal{E}({\mathbf{a}}|{\mathbf{a}}_{\infty}),
\end{equation*}
where ${\mathbf{a}}_{\infty}$ is the unique strictly
 positive complex balance equilibrium determined by $M,$
  and $\mathcal{E}({\mathbf{a}}|{\mathbf{a}}_{\infty})$ and $\mathcal{D}({\mathbf{a}})$ are defined in \eqref{ee1} and \eqref{ee2}, respectively.
		\par 
As a consequence, for any nonnegative initial data $\mathbf{a}_0$ with positive mass and finite relative entropy, i.e. $\sum_{i=1}^N \frac{1}{\alpha_i}\int_{\Omega} a_{i,0}(x)\, dx= M>0$ and
$\mathcal{E}({\mathbf{a}}_0| {\mathbf{a}}_{\infty}) < +\infty$,  the renormalised solutions of system \eqref{c_cycle_0} (as constructed in \cite{Fis15} and defined in the same manner as in \eqref{renormalised}), \eqref{Nin} (with ${\mathbf{c}}$ replaced by ${\mathbf{a}}$)
 converge exponentially fast in $L^1$-norm to the equilibrium ${\mathbf{a}}_{\infty}$:
		\begin{equation*}
			\sum_{i=1}^{N}\|a_i(t) - a_{i,\infty}\|_{L^1(\Omega)}^2 \leq C_{CKP}^{-1}e^{-\lambda t}\mathcal{E}(\mathbf a_0|\mathbf a_{\infty}) \quad \text{ for all } \quad t>0,
		\end{equation*}
		where $C_{CKP}$ is defined in Lemma \ref{CKPinequality} below.
	\end{theorem}

%
\bigskip 

{\bf The second part of this paper}, presented in Section \ref{sec:boundary}, is devoted to the study of complex balanced systems featuring boundary equilibria. 
We recall that even in the ODE setting \eqref{e00}, 
the convergence to the unique strictly positive complex balance equilibrium for general reaction networks with boundary equilibria involves the study of  the Global Attractor Conjecture. 

Moreover, we emphasise that due to the presence of boundary equilibria, 
we cannot expect an entropy entropy-dissipation estimate of the form  \eqref{main} to hold with the same generality 
as in Thm. \ref{EEDE}. 
In fact, it is easy to verify (see also Section \ref{sec:boundary}) that the entropy dissipation $\mathcal{D}(\cc)$ tends to zero for sequences of constant states, 
which converge to a boundary equilibrium and satisfy the mass conservation laws, yet the relative entropy to the positive complex balance equilibrium $\mathcal{E}(\cc|\ww)$ remains, of course, far from $0$ for such states. 
\medskip
  
In Section \ref{sec:boundary}, we prove two quite general results concerning the convergence to the unique strictly positive 
complex balance equilibrium 
for the following  two reaction-diffusion networks with boundary equilibria:
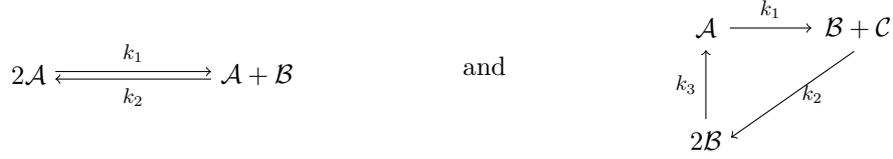
\begin{equation*}
\begin{tikzpicture}[baseline=(current  bounding  box.center)]
\node (a) {$2\mathcal{A}$} node (b) at (3,0) {$\mathcal{A} + \mathcal{B}$};
\draw[arrows=->] ([yshift=.5mm]a.east) --node[above]{\scalebox{.8}[.8]{$k_1$}} ([yshift=.5mm]b.west);
\draw[arrows=->] ([yshift=-.5mm]b.west) --node[below]{\scalebox{.8}[.8]{$k_2$}} ([yshift=-.5mm]a.east);
\end{tikzpicture}
\hspace{2cm}
\text{ and } 
\hspace{2cm}
\begin{tikzpicture} [baseline=(current  bounding  box.center)]
\node (a) {$\mathcal A$} node (b) at (2,0) {$\mathcal B+\mathcal C$} node (c) at (0,-1.5) {$2\mathcal B$};
\draw[arrows=->] ([xshift =0.5mm]a.east) -- node [above] {\scalebox{.8}[.8]{$k_1$}} ([xshift =-0.5mm]b.west);
\draw[arrows=->] ([yshift=0.5mm]c.north) -- node [left] {\scalebox{.8}[.8]{$k_3$}} ([yshift=-0.5mm]a.south);
\draw[arrows=->] ([xshift =-0.5mm,yshift=-0.5mm]b.south) -- node [right] {\scalebox{.8}[.8]{$k_2$}} ([yshift =0.5mm]c.east);
\end{tikzpicture}
\end{equation*}
The first system is described by the following $2\times 2$ mass action law reaction-diffusion system
\begin{equation}\label{2x2_0}
	\begin{cases}
		a_t - d_a\Delta a = -a^2 + ab, &x\in\Omega, \quad t>0,\\
		b_t - d_b\Delta b = a^2 - ab, &x\in\Omega, \quad t>0,\\
		\nabla a \cdot \nu = \nabla b \cdot \nu = 0, &x\in\partial\Omega, \quad t>0,\\
		a(x, 0) = a_0(x), \quad b(x,0) = b_0(x),\quad &x\in\Omega,
	\end{cases}
\end{equation}
where we assume normalised reaction rate constants $k_1 = k_2 = 1$ (without loss of generality since different rates can be recovered thanks to a suitable scaling).
\medskip

It is easily checked that the above $2\times 2$ system is complex balanced, and even 
detailed balanced, but features a boundary equilibrium.
 We nevertheless are able to prove
for this system the

\begin{proposition}[Convergence to detailed balance equilibrium for a $2\times 2$ system with boundary equilibrium]\label{2x2_theorem}\hfill\\
Let $\Omega$ be a bounded smooth ($C^2$) domain of ${\mathbb R}^n$ and $d_a,d_b>0$.
Assume that the initial data $(a_0, b_0) \in L^{\infty}(\Omega)\times L^{\infty}(\Omega)$ satisfy  $0<\varepsilon^2 \leq a_0(x), b_0(x) \leq \Lambda<+\infty$ (for a.a. $x\in\Omega$). 
	\par 
Then, the unique global classical ($C^2$ for $t>0$) solution $(a, b)$ to \eqref{2x2_0} 
conserves the total mass
\begin{equation*}
\int_{\Omega} (a(x,t) + b(x,t))\,dx = M := \int_{\Omega} (a_0(x) + b_0(x))\,dx >0\qquad \text{for all} \quad t>0,
\end{equation*}	
and 
converges exponentially fast to the unique strictly positive detailed balance equilibrium
	$(a_{\infty}, b_{\infty})$:
	\begin{equation*}\label{ExpoDecay2x2}
	\|a(\cdot,t) - a_{\infty}\|_{L^2(\Omega)}^2 + \|b(\cdot,t) - b_{\infty}\|_{L^2(\Omega)}^2 \leq e^{-\lambda_\varepsilon t}(\|a_0-a_{\infty}\|_{L^2(\Omega)}^2 + \|b_0 - b_{\infty}\|_{L^2(\Omega)}^2) \quad \text{ for all } \quad t>0,
	\end{equation*}
	where $\lambda_{\varepsilon}$ is an {\normalfont explicit constant} depending only on $\Omega, M, d_a, d_b$ and $\varepsilon$. Moreover, one can choose $\lambda_\varepsilon = O(\varepsilon^2)$ as $\varepsilon\rightarrow 0$.
\end{proposition}

This proposition can be used to get a result in the case when the initial data are not bounded below:

\begin{corollary}\label{2x2cor}
Let $\Omega$, $d_a$, $d_b$ be as in Prop. \ref{2x2_theorem}. 
Assume that the non-negative initial data $(a_0, b_0)$ belong to $C(\bar{\Omega})\times C(\bar{\Omega})$, and that $a_0$ is not the trivial initial state $0$.


Then, the unique global classical ($C^2$ for $t>0$) solution $(a, b)$ to \eqref{2x2_0} 
conserves the total mass (as in Prop. \ref{2x2_theorem}) and 
converges exponentially fast to the unique strictly positive detailed balance equilibrium, that is, for some $C, \lambda>0$,
\begin{equation*}\label{ExpoDecay2x2bis}
	\|a(\cdot,t) - a_{\infty}\|_{L^2(\Omega)}^2 + \|b(\cdot,t) - b_{\infty}\|_{L^2(\Omega)}^2 \leq C\, e^{-\lambda \, t} \quad \text{ for all } \quad t>0.
	\end{equation*}
\smallskip

Note that on the other hand, if $a_0\equiv0$ on $\Omega$ (so that the initial total mass is concentrated in $b_0$),
then the classical solutions to \eqref{2x2_0} converge exponentially fast to the boundary equilibrium $(0,M)$.
\end{corollary}
\medskip 

In the proof of Prop. \ref{2x2_theorem},
 the particular structure of 
 system \eqref{2x2_0} allows to show via a comparison principle argument that solutions 
propagate 
the assumed positive lower and upper $L^{\infty}$-bounds of the initial data.
As a consequence, the solutions remain bounded away from the boundary equilibrium.
Also the proof of Cor. \ref{2x2cor} applies the minimum principle. 
However, such comparison principle arguments (like the propagation of lower and upper $L^{\infty}$-bounds) 
are no longer true for more general systems such as the second example (described in the following).
\medskip 

Applying the mass action law to the second example leads to the below $3\times 3$ reaction-diffusion system
\begin{equation}\label{3x3_0}
	\begin{cases}
		a_t - d_a\Delta a = - k_1a + k_3b^2, &x\in\Omega, \quad t>0,\\
		b_t - d_b\Delta b = k_1a + k_2bc - 2k_3b^2, &x\in\Omega, \quad t>0,\\
		c_t - d_c\Delta c = k_1a - k_2bc, &x\in\Omega,  \quad t>0,\\
		\nabla a \cdot \nu = \nabla b \cdot \nu = \nabla c\cdot \nu = 0, &x\in\partial\Omega, \quad t>0,\\
		a(x,0) = a_0(x), \quad b(x,0) = b_0(x), \quad c(x,0) = c_0(x),\quad &x\in\Omega.
	\end{cases}
\end{equation}
\medskip

For system \eqref{3x3_0}, we first address the existence of global classical solution. Due to the lack of comparison principle arguments (which held for the $2\times 2$ system), we are able to show the existence of global classical solutions in dimensions $N\leq 5$, or, when $N\geq 6$, when the diffusion coefficients are assumed to be sufficiently close to each other (see Lemma \ref{ClassicalSolution}). The $3\times 3$ system \eqref{3x3_0} also features a boundary equilibrium. 
Since no  uniform-in-time upper or lower positive {\it{a priori}} $L^{\infty}$-bounds can be obtained for system \eqref{3x3_0},  
the question of instability of the boundary equilibrium
turns out be significantly more tricky.
\medskip

We can nevertheless prove the
\medskip 

\begin{theorem}[Convergence to complex balance equilibrium for a $3\times 3$ system with boundary equilibrium]\label{ExpConverge}\hfill\\
Let $\Omega$ be a bounded smooth ($C^2$) domain of ${\mathbb{R}}^n$, and $d_a,d_b,d_c>0$, $k_1, k_2, k_3 >0$.
Assume that the initial data $(a_0, b_0, c_0)\in (L^{\infty}(\Omega))^3$ are such 
 that $b_0$ is a.e. bounded below by a strictly positive constant, i.e. $\|\frac{1}{b_0}\|_{L^{\infty}(\Omega)}<\infty$. 
		\par 
	Then, there exist {\normalfont{explicit constants}} $C>0$  
	 and $\lambda>0$ such that any classical solutions of \eqref{3x3_0} converge exponentially fast to the unique strictly positive complex balance equilibrium $(a_{\infty}, b_{\infty}, c_{\infty})$ with the rate $\lambda$, that is
	\begin{equation*}\label{e_convergence}
		\|a(t) - a_{\infty}\|_{L^1(\Omega)}^2 + \|b(t) - b_{\infty}\|_{L^1(\Omega)}^2 + \|c(t) - c_{\infty}\|_{L^1(\Omega)}^2 \leq
		C\, e^{-\lambda t},
		 \qquad \text{for all}\quad t>0,
	\end{equation*}
where the constant $C$ depends only on the initial relative entropy and $C_{CKP}$, and $\lambda$ depends only on $\Omega$, $M$, $d_a$, $d_b$, $d_c$, $k_1$, $k_2$, $k_3$ and $\|\frac{1}{b_0}\|_{L^{\infty}(\Omega)}$.
	
\end{theorem}

In order to prove this result, we first observe that solutions to system \eqref{3x3_0},
which  initially satisfy $\|\frac{1}{b_0}\|_{L^{\infty}(\Omega)}<\infty$,  
obey a specific lower bound  of the form $\|\frac{1}{b(t)}\|_{L^{\infty}(\Omega)} \le C\, (1+t)$.  This lower bound 
shows that the dissipation of the relative entropy (with respect to the
unique strictly  positive equilibrium $(a_{\infty}, b_{\infty}, c_{\infty})$) decays at most like $(1+t)^{-1}$ times the 
relative entropy. 

As a consequence, we can show algebraically fast convergence to  $(a_{\infty}, b_{\infty}, c_{\infty})$ via a Gronwall argument. The essence of this Gronwall argument is illustrated by the following contradiction argument: Suppose that the relative entropy to  $(a_{\infty}, b_{\infty}, c_{\infty})$ would remain bounded below by a strictly positive constant uniformly-in-time, then 
the entropy dissipation would also remain bounded below by a positive constant times 
$(1+t)^{-1}$. The function $(1+t)^{-1}$ is however non-integrable over the interval $[0,+\infty)$ and thus
yields a contradiction to the fact that the time-integral  of the entropy dissipation over $[0,+\infty)$ is bounded by the initial relative entropy.


Moreover in a second step, the obtained algebraic convergence allows to recover exponential convergence towards the strictly 
positive complex balance equilibrium $(a_{\infty}, b_{\infty}, c_{\infty})$. In fact, the rate of the convergence can be  estimated explicitly as in
Thm. \ref{c_lem:eede} (see Rmk. \ref{rmk:explicit}).

\begin{remark}[Convergence to boundary equilibria]
The condition on $b_0$ in Thm \ref{ExpConverge} is clearly stronger than the assumption that the nonnegative initial datum $(a_0, b_0, c_0)$ is different from the boundary equilibrium. In fact, there is a class of initial data (strictly larger than the boundary equilibrium) for which solutions converge to the boundary equilibrium, see 
Rmk. \ref{r44}. Such a class of initial data was also presented in Cor. \ref{2x2cor} for the $2\times 2$ system.
Moreover, Cor. \ref{2x2cor} allowed to precisely classify for continuous initial data, which solutions converge to the boundary equilibrium and which solutions converge to the positive complex balance equilibrium. 
\end{remark}

\begin{remark} 
	We would like to point out a significant difference between our  method here and the linearisation techniques: when using  linearisation, one needs to wait until the trajectory is trapped into a small enough neighbourhood of $\ww = (a_{\infty}, b_{\infty}, c_{\infty})$.
 Our method, however, 
already allows to use the standard entropy method (and obtain exponential convergence to $\ww$)
as soon as the solution is not close to the boundary equilibria.

Note also that the arguments used in the $3\times 3$ system are different from the ones used in the
so-called  ``slowly growing {\it{a priori}} bounds'' variant of the entropy method
 (cf. \cite{tovi}, \cite{DeFe08}).
 \end{remark}
 

\medskip

The main novelties of this work are the following:
\begin{itemize}
	\item[(i)] Up to our knowledge, this paper seems to be the first quantitative study of convergence to equilibrium for 
nonlinear complex balanced reaction-diffusion systems.
	
%
Previous related results on the large time behaviour of complex balanced systems in the ODE setting
stated mostly qualitative results, that is convergence to equilibrium, but without convergence rates, see e.g. \cite{And11, CNN13, GMS14, Pan12} and references therein.
The only quantitative results for ODE complex balanced systems known to us was proven by \cite{SJ08}, in which the authors obtained convergence rates close to equilibrium via a linearisation technique.

The PDE setting \eqref{e0} for complex balanced systems is even less studied. Concerning the convergence to equilibrium for detailed balanced systems in the context of semi-conductor models, we refer to \cite{Gro92, GGH96, GH97},  where convergence rates were proven, 
but with non-constructive constants coming out of an abstract compactness argument. In \cite{DeFe06}, the authors obtained the first results of convergence to equilibrium for nonlinear detailed balance reaction-diffusion systems with explicit rates and constants.
 Related results were later derived in e.g. \cite{DeFe08, GZ10, DFIFIP, MiHaMa14,FT15}. We point out that all these references consider reaction-diffusion systems which satisfy a detailed balance condition. 
  Finally, a recent result on 
the quantitative convergence to equilibrium was proven in \cite{FPT15} 
in the special case of linear complex balanced systems for PDEs (coming out of  first order chemical reaction networks).

	\medskip
	\item[(ii)] Secondly, we study two specific complex balanced systems featuring boundary equilibria.
	
	One of the main difficulties in proving the Global Attractor Conjecture for ODEs is the appearance of boundary equilibria,  leading to the possibility of $\omega$-limit sets having nonempty intersections with the boundary $\partial\mathbb{R}^N_{\geq 0}$. Resolving this problem in the PDE setting \eqref{e0} is, up to the best of our knowledge, completely open. Here, we investigate two particular systems featuring boundary equilibria and study in which modified sense the GAC could be investigated: First, we note that in the PDE setting, there exist nontrivial classes of solutions which converge to the boundary equilibrium. Secondly, we point out the observation that for some specific examples of complex balanced systems, the nonlinear reaction terms which gives rise to boundary equilibriua, 
also lead to specifically decaying {\it{a priori}} lower bounds which (in combination with an entropy dissipation argument) are sufficient to prove instability of those boundary equilibria.
\end{itemize}

\medskip
\noindent{\bf Outline:}
The rest of the paper is organised as follows: 
In Section \ref{sec:noboundary} contains the results concerning complex balanced systems without boundary equilibria. 
The two specific systems featuring boundary equilibria are then analysed in Section \ref{sec:boundary}. 

\medskip
\noindent{\bf Notations:}
With a slight abuse of notation, we recall the following convention whenever a single letter $y$ is used for a complex, then $y_i$ (for $i=1,\ldots, N$) denotes the stoichiometric coefficient of the $i$-th specie $S_i$ in the complex $y$. If the complex is written as $y_r$ (or $y'_r$), then it denotes the source (or target) complex of the $r$-th reaction, and the $i$-th stoichiometric coefficient is denoted by $y_{r,i}$ (or $y'_{r,i}$).
 
\begin{itemize}[topsep=5pt, leftmargin=8mm]
	\item For $\cc, y \in \mathbb R^N$, we denote 
	\begin{equation*}
		\cc^y = \prod_{i=1}^{N}c_i^{y_i}\quad \text{ and } \quad \frac{\cc}{y} = \left(\frac{c_1}{y_1}, \frac{c_2}{y_2}, \ldots, \frac{c_N}{y_N}\right) \in \mathbb{R}^N.
	\end{equation*}
	The scalar product is denoted by
	\begin{equation*}
		\cc \cdot y = \sum_{i=1}^{N}c_iy_i.
	\end{equation*}
	\item For a function $f: \mathbb R \rightarrow \mathbb R$ and $\cc \in \mathbb R^N$, we denote 
	\begin{equation*}
		f(\cc) = \left(f(c_1), f(c_2), \ldots, f(c_N)\right) \in \mathbb R^N.
	\end{equation*}
	\item We denote by $\|\cdot\|$ the usual norm in $L^2(\Omega)$.
	\item  We can assume (w.l.o.g. by rescaling the space variable $x$) that $\Omega$ has normalised volume
$$
|\Omega| = 1.
$$
Note that we shall systematically make this assumption in all the following proofs of the theorems concerning PDEs.
\item In Section \ref{sec:noboundary}, we shall use the notations $\cc$, $\overline{\cc}$, $\ww$, etc. whenever a result is proven for general complex 
balanced systems without boundary equilibrium. On the other hand, for results specific to a cycle of reactions and systems 
\eqref{c_cycle_0}, we shall use the notations ${\mathbf{a}}$, ${\mathbf{\bar{a}}}$, ${\mathbf{a}}_{\infty}$, etc. instead.
In this case, it is also convenient to use the periodic notation $a_{N+1} := a_1$ or $a_{N+1,\infty} := a_{1,\infty}$ e.t.c.
Moreover, we shall introduce capital letters as 
a short hand notation for square roots of quantities, e.g. $C_i := \sqrt{c_{i}}$, $\overline{C_{i}} := \overline{\sqrt{c_{i}}}$, $C_{i,\infty} := \sqrt{c_{i,\infty}}$, $A_i := \sqrt{a_{i}}$, etc. 
\end{itemize}

\section{Complex balanced systems without boundary equilibria}\label{sec:noboundary}

 Before beginning the proof of Thm. \ref{EEDE}, we first state and prove a proposition devoted to the study 
 of ODEs coming out of a chemical reaction netwok, since the tools developed in this proof can then be used 
 in the study of the PDEs.

\subsection{Convergence to equilibrium for ODE systems}\label{convergenceODE}\hfill\\
Throughout this paper, we frequently use the two variables function $\Psi$ defined by \eqref{psi},
and the elementary estimate
\begin{equation} \label{elest}
	\Psi(x,y) \geq (\sqrt{x} - \sqrt{y})^2.
\end{equation}

It is well known 
for reaction terms ${\mathbf{R}}$ describing complex balanced chemical networks without boundary equilibrium that  solutions $\cc := \cc (t)$ of the ODE system \eqref{e00} (with
nonnegative initial data $\cc_0\ge0$ corresponding to a strictly positive mass vector $M\in \mathbb{R}^m_{>0}$) converge towards  
the unique strictly positive complex balance equilibrium $\ww$.
Moreover, the relative entropy functional
\begin{equation}\label{ODE_Entropy}
E(\cc(t)|\ww) := \sum_{i=1}^{N}\Psi(c_i(t); c_{i,\infty}) = \sum_{i=1}^{N}\left(c_i(t)\log{\frac{c_i(t)}{c_{i,\infty}}} - c_i + c_{i,\infty}\right)
\end{equation}
is decreasing along solution trajectories (see e.g. \cite{HoJa72, Fe79, And14, Gop13}). 
The corresponding entropy dissipation, however, seems not to be written down. We provide the following proposition to 
establish the explicit expression of the entropy dissipation:
\medskip

\begin{proposition}\label{cale}
We consider a reaction rate defined by \eqref{Reaction} and $\ww$ a strictly positive complex balance  equilibrium of a chemical reaction network (cf. Defs. \ref{def:crn} and \ref{Equilibria}).
Then,
\begin{equation}\label{33bis}
D(\cc) := -\RR(\cc)\cdot \log{\frac{\cc}{\ww}} = \sum_{r=1}^{|\mathcal{R}|}k_r\ww^{y_r}\,\Psi\left(\frac{\cc^{y_r}}{\ww^{y_r}};\frac{\cc^{y_r'}}{\ww^{y_r'}}\right) \ge 0.
\end{equation}
Moreover $D(\cc_*) =0$ for some $\cc_*\in \mathbb{R}^N_{\geq 0}$ if and only if $\cc_*$ is a complex balanced equilibrium (that is, \eqref{ComplexBalance} holds). Finally, if the chemical reaction network has no 
boundary equilibria, then $\cc_* = \ww$.
\end{proposition}

\begin{proof}
We compute
\begin{equation}\label{e2}
\begin{aligned}
-\RR(\cc)\cdot \log{\frac{\cc}{\ww}} &= \Biggl(\sum_{r=1}^{|\R|}k_r\cc^{y_r}(y_r - y_r')\Biggr)\cdot \log{\frac{\cc}{\ww}}
= \sum_{r=1}^{|\R|}k_r\cc^{y_r}\log{\frac{\cc^{y_r - y_r'}}{\ww^{y_r - y_r'}}}\\
				&= \sum_{r=1}^{|\R|}\left[k_r\ww^{y_r}\left(\frac{\cc^{y_r}}{\ww^{y_r}}\log\left(\frac{\cc^{y_r}}{\ww^{y_r}}\bigg/\frac{\cc^{y_r'}}{\ww^{y_r'}}\right) - \frac{\cc^{y_r}}{\ww^{y_r}} + \frac{\cc^{y_r'}}{\ww^{y_r'}}\right) + k_r\cc^{y_r} - k_r\cc^{y_r'}\frac{\ww^{y_r}}{\ww^{y_r'}}\right]\\
				&= \sum_{r=1}^{|\R|}k_r\ww^{y_r}\Psi\left(\frac{\cc^{y_r}}{\ww^{y_r}}; \frac{\cc^{y_r'}}{\ww^{y_r'}}\right) + \sum_{r=1}^{|\R|}\left[k_r\cc^{y_r} - k_r\cc^{y_r'}\frac{\ww^{y_r}}{\ww^{y_r'}}\right].
	\end{aligned}
\end{equation}
It remains to prove that
\begin{equation}\label{e3}
	\sum_{r=1}^{|\R|}\left[k_r\cc^{y_r} - k_r\cc^{y_r'}\frac{\ww^{y_r}}{\ww^{y_r'}}\right] = 0.
\end{equation}
Indeed, by using the properties of Def. \ref{def:crn},
\begin{equation}\label{e4}
\begin{aligned}
	\sum_{r=1}^{|\R|}\left[k_r\cc^{y_r} - k_r\cc^{y_r'}\frac{\ww^{y_r}}{\ww^{y_r'}}\right] 
	&= \sum_{y\in \C}\Biggl[\,\sum_{\{r:\, y_r = y\}}k_r\cc^{y_r} - \sum_{\{s:\,y_s'=y\}}k_s\cc^{y_s'}\frac{\ww^{y_s}}{\ww^{y_s'}}\Biggr]\\
	&=  \sum_{y\in \C}\Biggl[\,\cc^{y}\sum_{\{r:\, y_r = y\}}k_r - \frac{\cc^{y}}{\ww^{y}}\sum_{\{s:\,y_s'=y\}}k_s\ww^{y_s}\Biggr]\\
	&= \sum_{y\in \C}\frac{\cc^{y}}{\ww^{y}}\Biggl[\,\sum_{\{r:\, y_r = y\}}k_r\ww^{y_r} - \sum_{\{s:\,y_s'=y\}}k_s\ww^{y_s}\Biggr]= 0,
\end{aligned} 
\end{equation}
where we have used the complex balance condition \eqref{ComplexBalance} in the last step.
		
Now assume that $D(\cc_*) = 0$ for some $\cc_*\in \mathbb{R}^N_{\geq 0}$. Since $\Psi(x,y) \geq 0$ for all $x\ge 0$, $y >0$ and $\Psi(x,y) = 0$ if and only if $x = y$, it follows from $D(\cc_*) = 0$ that
		\begin{equation}\label{e8}
			\frac{\cc_*^{y_r}}{\ww^{y_r}} = \frac{\cc_*^{y_r'}}{\ww^{y_r'}} \quad \text{ or equivalently } \quad \frac{\cc_*^{y_r'}}{\ww^{y_r'}}\ww^{y_r} = \cc_*^{y_r}
		\end{equation}
		for all $r = 1,\ldots, |\mathcal R|$. Thus for any $y\in \mathcal{C}$, we have
		\begin{equation*}
									\sum_{\{r:\,y_r=y\}}k_r\cc_*^{y_r}  = \frac{\cc_*^{y}}{\ww^y}\sum_{\{r:\,y_r=y\}}k_r\ww^{y_r}
									=\frac{\cc_*^{y}}{\ww^y}\sum_{\{s:\, y_s' = y\}}k_s\ww^{y_s}
									= \sum_{\{s:\,y_s' = y\}}k_s\frac{\cc_*^{y_s'}}{\ww^{y_s'}}\ww^{y_s}
								    = \sum_{\{s:\,y_s' = y\}}k_s\cc_*^{y_s}.
		\end{equation*}
Therefore, $\cc_*$ is a complex balance equilibrium due to \eqref{ComplexBalance}. Furthermore, if the considered
chemical reactions network 
 has no boundary equilibria, then obviously $\cc_* \equiv \ww$.
\end{proof}

\begin{remark}[Explicit Entropy-Dissipation]\label{Rem:Explicit}\hfill\\
	In the case of linear reaction networks, a recent result \cite{FPT15} shows explicitly that the relative entropy between any two solutions decays. Such a strong results does not seem true for nonlinear reaction networks, for which we suspect that only the relative entropy \eqref{ODE_Entropy} is a Lyapunov functional in general. For linear systems, the existence of a dissipative relative entropy between any two solutions is due to a fundamental property of linear systems, see e.g. \cite{FJ16}. It can also be derived from the General Entropy Principle for finite Markov chains, see \cite{PMP06}.
\end{remark}

 Thanks to the explicit expression of the entropy dissipation in Prop. \ref{cale}, we 
 are now in  position to prove the following quantitative result about the decay towards equilibrium of solutions of the ODE systems \eqref{e00} (for complex balanced chemical reactions networks without boundary equilibria).
 \medskip

\begin{proposition}[Exponential convergence to equilibrium for ODE systems without boundary equilibria]\label{ODE_Theorem}\hfill\\
We consider a complex balanced chemical reaction network without boundary equilibria, the corresponding reaction term 
\eqref{Reaction} and, for a given mass vector $M\in\mathbb{R}^m_{>0}$, the unique strictly positive complex balanced equilibrium $\ww$
together with the ODE system \eqref{e00}. 
	
Then, there exists a strictly positive constant $\lambda>0$ depending only on the coefficients appearing in ${\mathbf{R}}$, such that 
 any solution $\cc := \cc(t)$ to the ODE system \eqref{e00} (with nonzero nonnegative initial data)
  satisfies the following exponentially fast decay of the relative entropy:
	\begin{equation}\label{ODE_decay}
		E(\cc(t)|\ww) \leq E(\cc_0|\ww)\, e^{-\lambda t} \quad \text{ for all } \quad t>0.
	\end{equation}
As a consequence, these solutions $\cc$ converge exponentially fast to the unique strictly positive complex balance equilibrium $\ww$:
	\begin{equation}\label{Converge_ODE}
		\sum_{i=1}^{N}|c_i(t) - c_{i,\infty}|^2 \leq C E(\cc_0|\ww)e^{-\lambda t} \quad \text{ for all } \quad t>0 ,
	\end{equation}
	where $C = (\sqrt{\tilde{K}} + \max_{i}\{\sqrt{c_{i,\infty}}\})^2$ and $\tilde{K}$ is defined as in \eqref{KK}.
\end{proposition}
\begin{proof}

We first observe that
\begin{equation}\label{e1}
-\frac{d}{dt}E(\cc|\ww) = -\sum_{i=1}^{N}\left(\frac{d}{d t}c_i\right)\log{\frac{c_i}{c_{i,\infty}}}= - \RR(\cc)\cdot \log{\frac{\cc}{\ww}},
\end{equation}
so that thanks to \eqref{33bis},  
\begin{equation}\label{edla}
-\frac{d}{dt} E(\cc|\ww) = D(\cc)
= \sum_{r=1}^{|\mathcal{R}|}k_r\ww^{y_r}\,\Psi\left(\frac{\cc^{y_r}}{\ww^{y_r}};\frac{\cc^{y_r'}}{\ww^{y_r'}}\right). 
\end{equation}  
Thus, if we are able to show the following {\it entropy entropy-dissipation estimate}
	\begin{equation}\label{inequal}
		D(\cc(t)) \geq \lambda\, E(\cc(t)|\ww),
	\end{equation}
for some $\lambda >0$ and for all $t>0$, we 
see that \eqref{ODE_decay} will follow from 
 a Gronwall argument.
	We consider
		\begin{equation}\label{CM}
			\mathfrak{C}_M = \left\{\xi\in \mathbb R_+^N: \mathbb Q\, \xi = M \quad \text{ and } \quad E(\xi|\ww) \leq E(\cc_0|\ww) \right\}.
		\end{equation}
		Thanks to \eqref{edla}, we see that $\cc(t) \in \mathfrak{C}_M$ for all $t>0$. Therefore, we will prove the estimate (\ref{inequal}) thanks to the
		
		\begin{lemma}\label{intermediate}
			Let $E(\cdot|\ww)$, $D(\cdot)$ and $\mathfrak{C}_M$ be defined in \eqref{ODE_Entropy}, \eqref{33bis} and \eqref{CM} respectively. Then
			\begin{equation}\label{lambda}	
					\lambda := \inf_{\xi\in \mathfrak{C}_M}\frac{D(\xi)}{E(\xi|\ww)} >0.
				\end{equation}
		\end{lemma}
		\begin{proof}
		Firstly, since $D(\xi) \geq 0$ and $E(\xi|\ww) \geq 0$, we see that $\lambda  \geq 0$. Secondly, it follows from  definition \eqref{CM} of $\mathfrak{C}_M$ that
		 the denominator $E(\xi|\ww)$ is bounded above. On the other hand, thanks to Prop. \ref{cale}, $D(\xi) = 0$ if and only if $\xi = \ww$. Thus, $\lambda $ can approach zero only in cases when $\xi \rightarrow \ww$ on $\mathfrak{C}_M$ (where we recall that $\ww$ is strictly positive). 
		Therefore, in order to prove \eqref{lambda}, it is sufficient to show that
		\begin{equation}\label{e9}
			\liminf_{\mathfrak{C}_M \ni \xi \rightarrow \ww}\frac{D(\xi)}{E(\xi|\ww)} > 0.
		\end{equation}
	Thanks to $D(\ww) = 0$, $\nabla D(\ww) = 0$, and
	\begin{equation*}
		(\xi - \ww)^T\nabla^2D(\ww)(\xi - \ww) = \sum_{r=1}^{|\mathcal R|}k_r\left(\frac{y_{r,i} - y_{r,i}'}{c_{i,\infty}}(\xi - c_{i,\infty})\right)^2,
	\end{equation*}
which follows from direct computations, 
	we can Taylor expand  around $\ww$ to obtain
		\begin{equation}\label{e10}
			\Lambda:= \liminf_{\mathfrak{C}_M \ni \xi \rightarrow \ww}\frac{{D}(\xi)}{{E}(\xi|\ww)} = \inf_{\xi \in \mathfrak{C}_M}\left\{\frac{\sum\limits_{r=1}^{|\R|}k_r\left(\sum\limits_{i=1}^{N}\frac{y_{r,i} - y_{r,i}'}{c_{i,\infty}}(\xi_i - c_{i,\infty})\right)^2}{\sum_{i=1}^{N}\frac{(\xi_i - c_{i,\infty})^2}{c_{i,\infty}}}\right\}.
		\end{equation}
		Denoting $\boldsymbol{\mu} = \xi - \ww$, we see that \eqref{e10} becomes
					\begin{equation}\label{e10+1}
						\Lambda  = \inf_{\boldsymbol{\mu} + \ww \in \mathfrak{C}_M}\left\{\frac{\sum\limits_{r=1}^{|\R|}k_r\left(\sum\limits_{i=1}^{N}\frac{y_{r,i} - y_{r,i}'}{c_{i,\infty}}\mu_i\right)^2}{\sum_{i=1}^{N}\frac{\mu_i^2}{c_{i,\infty}}}\right\}.
					\end{equation}
	Since both numerator and denominator of the above fraction are of homogeneity two, 
	 it is sufficent to estimate \eqref{e10+1} for $\boldsymbol{\mu}$ in the unit ball, i.e. $|\boldsymbol{\mu}|=1$. By putting $\eta_i = \frac{\mu_i}{c_{i,\infty}}$, we have $\boldsymbol{\mu} = \mathrm{diag}(\ww)\,\boldsymbol{\eta}$. Hence, the numerator of \eqref{e10+1} becomes
					\begin{equation}\label{e11}
						\sum_{r=1}^{|R|}k_r\left((y_r - y_r')\cdot \boldsymbol{\eta}\right)^2.
					\end{equation}
					This term can only be zero when $\boldsymbol{\eta}\in \mathrm{ker}(W)$, where we recall that $W$ is the Wegscheider matrix
	\begin{equation}\label{WegscheiderMatrix}	
	{W} = [(y_r' - y_r)_{r=1,\ldots,|\R|}]^{\top} \in \mathbb{R}^{|\R|\times N}.
	\end{equation}
		We now consider two cases, related to the number of conservation laws $m =  \mathrm{dim}\,\mathrm{ker}(W)$:
		\begin{description}
			\item[Case 1. $m>0$] Since the rows of $\mathbb{Q}$ form a basis of $\mathrm{ker}(W)$, there exists $\mathbf a\in \mathbb{R}^m$ such that
				\begin{equation}\label{e12}
				\boldsymbol{\eta} = \mathbb{Q}^{\top}\mathbf{a}.
				\end{equation}
				From $\mathbb{Q}\, \boldsymbol{\mu} = \mathbb{Q}\, \xi - \mathbb{Q}\, \ww= M - M = 0$, we get
	\begin{equation}\label{e13}
		\mathbb{Q}\,\mathrm{diag}(\ww)\,\mathbb{Q}^{\top} \mathbf{a} = \mathbb{Q}\,\mathrm{diag}(\ww)\boldsymbol{\eta} = \mathbb{Q}\boldsymbol{\mu}= 0.
		\end{equation}
				It follows from $\mathrm{rank}(\mathbb Q) = m$ and $c_{i,\infty}>0$ (for all $i=1,2,\ldots, N$) that $\mathbb{Q}\,\mathrm{diag}(\ww)\,\mathbb{Q}^{\top}$ is a positive definite $m \times m$ matrix. Hence, \eqref{e13} implies $\mathbf{a} = 0$,
				 which shows $\boldsymbol{\eta} = 0$, and consequently $\boldsymbol{\mu} = 0$,
				  which contradicts the fact that $|\boldsymbol{\mu}| = 1$.
				\medskip 
				
				
			\item[Case 2. $m=0$] In this case it is obvious that $\Lambda = 0$ if and only if $\boldsymbol{\eta} = 0$,  since $\mathrm{ker}(W) = 0$.
		\end{description}
		 In conclusion, we have shown that $\Lambda >0$ and consequently \eqref{lambda}. 
		 \end{proof}
	 
	 A direct consequence of Lemma \ref{intermediate} is that \eqref{inequal} holds and thus \eqref{ODE_decay} holds thanks to a classic Gronwall lemma. 
	 From $E(\cc(t)|\ww) \leq E(\cc_0|\ww)$ and the inequality $\Psi(x;y) \geq (\sqrt{x} - \sqrt{y})^2$ we get
	 \begin{equation*}
	 	E(\cc_0|\ww) \geq \sum_{i=1}^{N}\Psi(c_i(t); c_{i,\infty})  \geq \sum_{i=1}^{N}\left(\sqrt{c_i(t)} - \sqrt{c_{i,\infty}}\right)^2 \geq \frac{1}{2}\sum_{i=1}^{N}c_i(t) - \sum_{i=1}^{N}c_{i,\infty}.
	 \end{equation*}
	 Thus for all $i=1,\ldots, N$ and all $t>0$,
	 \begin{equation}\label{KK}
	 	c_i(t) \leq \tilde{K}:= 2\left(E(\cc_0|\ww) + \sum_{i=1}^{N}c_{i,\infty}\right).
	 \end{equation}
	 We can therefore show that
	 \begin{align*}
	 	E(\cc(t)|\ww) &\geq \sum_{i=1}^{N}\left(\sqrt{c_i(t)} - \sqrt{c_{i,\infty}}\right)^2 = \sum_{i=1}^{N}\frac{(c_i(t) - c_{i,\infty})^2}{(\sqrt{c_i(t)} + \sqrt{c_{i,\infty}})^2} \\
		&\geq \frac{1}{(\sqrt{\tilde{K}} + \max_{i}\{\sqrt{c_{i,\infty}}\})^2}\sum_{i=1}^{N}|c_i(t) - c_{i,\infty}|^2
	 \end{align*}
	 and finally get \eqref{Converge_ODE}.
%
\end{proof}

\subsection{Convergence to equilibrium for PDE systems}\label{convergenceconvex}\hfill\\
We now begin the

\begin{proof}[Proof of Theorem \ref{EEDE}]

We now consider $\mathcal{E}(\cc|\ww)$ given by formula \eqref{FreeEnergy_PDE} and 
$\mathcal{D}(\cc)$ given by \eqref{ED_PDE} as stated in the introduction.
 \medskip
 
 	Note that  in the PDE setting, the entropy entropy-dissipation estimate \eqref{main} is a functional inequality in contrast to the finite dimensional inequality \eqref{lambda} in the ODE setting. 
Thus, the techniques used in Proposition \ref{ODE_Theorem} are necessary (since \eqref{main} includes the ODE case) yet far from sufficient.
In order to prove \eqref{main},
we first rewrite its right hand side by using the additivity property of the relative entropy, that is (remembering that $\overline{c_i} = \int_{\Omega} c_i$ and $\ww$
do not depend upon $x$):
	\begin{equation*}
			\mathcal{E}(\cc|\ww) = \sum_{i=1}^{N}\int_{\Omega}c_i\log{\frac{c_i}{\overline{c_i}}}dx + \sum_{i=1}^{N}\left(\overline{c_i}\log{\frac{\overline{c_i}}{c_{i,\infty}}} - \overline{c_i} + c_{i,\infty}\right)
			=\mathcal{E}(\cc|\overline{\cc}) + \mathcal{E}(\overline{\cc}|\ww).
	\end{equation*}
	In order to treat the first term $\mathcal{E}(\cc|\overline{\cc})$ in the above identity, we apply the Logarithmic Sobolev Inequality
	\begin{equation*}
		\int_{\Omega}\frac{|\nabla f|^2}{f}dx \geq C_{LSI}\int_{\Omega}f\log{\frac{f}{\overline{f}}}dx
	\end{equation*}
	to estimate
	\begin{equation}\label{lam1}
		\frac{1}{2}\mathcal{D}(\cc) \geq \frac12 \sum_{i=1}^N \int_{\Omega}d_i\frac{|\nabla c_i|^2}{c_i}dx \geq \lambda_1\mathcal{E}(\cc|\overline{\cc}),
	\end{equation}
	with $\lambda_1 = \frac 12 C_{LSI}\min_{i}\{d_i\}$.
	By denoting 
	\begin{equation}\label{322bis}
		G(\cc) := \frac{1}{2}\sum_{r=1}^{|\mathcal R|}k_r\ww^{y_r}\Psi\left(\frac{\cc^{y_r}}{\ww^{y_r}};\frac{\cc^{y_r'}}{\ww^{y_r'}}\right) \quad \text{ and } \quad F(\cc) := \lambda_1\sum_{i=1}^{N}\Psi(c_i;c_{i,\infty}),
	\end{equation}
	we estimate (using again the Logarithmic Sobolev Inequality and the additivity of the relative entropy) 
	\begin{align}\label{tt1}
		\frac{1}{2}\mathcal{D}(\cc) \geq \lambda_1\mathcal{E}(\cc|\overline{\cc}) + \int_{\Omega}G(\cc)dx &= \int_{\Omega}\underbrace{\left[F(\cc) + G(\cc)\right]}_{:= \Phi(\cc)}dx - F(\overline{\cc}) \nonumber \\
		&=: \int_{\Omega}\Phi(\cc)\,dx - F(\overline{\cc}).
	\end{align}
	Let $\widehat\Phi$ be the convexification of $\Phi$, that is,
	 $\widehat \Phi$ is the supremum of all affine functions below $\Phi$.
	 By noticing that $\Phi \ge \widehat{\Phi}$ and that $\widehat{\Phi}$ is convex,
	  we can now use Jensen's inequality (recall that $|\Omega| = 1$, as stated in the paragraph devoted to notations at the end of the introduction)
	   and get
	\begin{equation}\label{lam2}
		\frac{1}{2}\mathcal{D}(\cc) \geq \int_{\Omega}\widehat{\Phi}(\cc)\,dx - F(\overline{\cc}) \geq \widehat \Phi(\overline{\cc}) - F(\overline{\cc}).
	\end{equation}	
		It remains to show the following finite-dimensional inequality:
		\begin{equation*}
			\widehat{\Phi}(\overline{\cc}) - F(\overline{\cc}) \geq \lambda_2\, \mathcal{E}(\overline{\cc}|\ww),
		\end{equation*}
		for some positive constant $\lambda_2>0$. By using Jensen's inequality, we see that 
		\begin{equation*}
			\mathcal{E}(\overline{\cc}|\ww) \leq \mathcal{E}(\cc|\ww) \leq \mathcal{E}(\cc_0|\ww).
		\end{equation*} 
		We define, for some given $K>0$,
		\begin{equation*}			
			\mathfrak C_M := \{\boldsymbol{\xi}\in \mathbb R^N_{+}: \mathbb Q\, \boldsymbol{\xi} = M \quad \text{ and } \quad \mathcal{E}(\boldsymbol{\xi}|\ww) \leq K
			\}.
		\end{equation*}
	Our goal is now to prove the
	\begin{lemma}\label{lem:PDEfinite}
	\begin{equation}\label{lambda2}
		\lambda_2:= \inf_{\boldsymbol{\xi} \in \mathfrak{C}_M}\frac{\widehat \Phi({\boldsymbol{\xi}}) - F(\boldsymbol{\xi})}{\mathcal{E}({\boldsymbol{\xi}}|\ww)} >0.
	\end{equation}
	\end{lemma}
	\begin{proof}
	We first observe that $\lambda_2\geq 0$. Indeed, using the
	inequality 
	 $\widehat{f + g} \geq \widehat{f} + \widehat{g}$ (see e.g. \cite{WCHL96}), we have
	\begin{equation*}
		\widehat{\Phi}(\boldsymbol\xi) - F(\boldsymbol{\xi}) \geq \widehat{F}(\boldsymbol\xi) + \widehat{G}(\boldsymbol\xi) - F(\boldsymbol\xi) = \widehat{G}(\boldsymbol\xi) \geq 0,
	\end{equation*}
	thanks to the convexity of $F$ (and the nonnegativity of $G$). On the other hand, since $\mathcal{E}(\boldsymbol\xi|\ww)$ is bounded above in $\mathfrak C_M$, then
	 $\lambda_2$ can be equal to zero only when 
	$$
		\widehat \Phi(\boldsymbol\xi) - F(\boldsymbol\xi) \rightarrow 0 \quad \text{ or equivalently } \quad \widehat{G}(\boldsymbol\xi) \rightarrow 0.
	$$
	Now using the fact that (when $\mathbb Q\,\boldsymbol\xi = M$)  $G(\boldsymbol\xi) = 0$ if and only if $\boldsymbol\xi = \ww$, it follows that $\widehat{G}(\boldsymbol\xi) = 0$ if and only if $\boldsymbol\xi = \ww$, thanks to the definition of the convexification.
	 Therefore $\lambda_2$ could only vanish for states $\mathfrak C_M \ni \boldsymbol\xi \rightarrow \ww$. 
	In other words, it is sufficient to prove that
	\begin{equation}\label{h1}
		\liminf_{\mathfrak{C}_M\ni \boldsymbol\xi\rightarrow \ww}\frac{\widehat{\Phi}(\boldsymbol\xi) - F(\boldsymbol\xi)}{\mathcal{E}(\boldsymbol\xi|\ww)} > 0.
	\end{equation}
We then use the
\medskip 

\begin{lemma}\label{lem:conv}\cite{MiHaMa14}\hfill\\
	Let $\Phi$ be defined by \eqref{tt1} and \eqref{322bis}. 
	
	Then, there exists $\delta>0$ such that for all $\boldsymbol\xi\in \mathbb{R}_{>0}^N\cap B(\ww, \delta)$ (ball centered at $\ww$ with radius $\delta$) the identity
	 $\widehat \Phi(\boldsymbol\xi) = \Phi(\boldsymbol\xi)$ holds.
\end{lemma}
\begin{proof}
The proof of this Lemma is given in \cite[Lemma 3.4]{MiHaMa14}. 
For the sake of completeness and the convenience of the reader, we also provide a proof in the Appendix of this paper. 

We point out that all estimates in the proof of Lemma \ref{lem:conv} are explicit. This means that
it is in principle possible to actually estimate the radius $\delta >0$ of the ball $B(\ww, \delta)$, on which  $\Phi$ and its convexification $\widehat{\Phi}$ are identical.
\end{proof}

	We continue the proof of Lemma \ref{lem:PDEfinite}. Thanks to Lemma \ref{lem:conv} above, we see that $\widehat{\Phi} \equiv \Phi$ in a neighborhood of $\ww$. Thus, \eqref{h1} is equivalent to 
	\begin{equation}\label{h2}
		\liminf_{\mathfrak{C}_M\ni \boldsymbol\xi \rightarrow \ww}\frac{G({\boldsymbol\xi})}{\mathcal{E}({\boldsymbol{\xi}}|\ww)} >0 ,
	\end{equation}	
which was already proven in Proposition \ref{ODE_Theorem} (see \eqref{e9} and note that $D$ in \eqref{e9} corresponds to $G$ in \eqref{h2}), up to a small modification ($E(\cc_0 | \ww)$ is replaced by some arbitrary $K>0$). 
	\end{proof}
	Therefore, from \eqref{lam1}, \eqref{lam2} and \eqref{lambda2}, we obtain the main estimate \eqref{main} with $\lambda  = \min\{\lambda_1, \lambda_2\}$.
\medskip

We first note that for the renormalised solutions considered in our theorem (that is, satisfying 
\eqref{integrability}), the entropy inequality 
$ \mathcal E(\bar{\cc}(t)|\ww) \le \mathcal E(\cc_0|\ww)$ holds, so that we can use  
 estimate \eqref{main}.
 We can then directly apply an integral form of Gronwall's lemma, see e.g. \cite{Wil} or more specifically \cite{FL15}, to finally ensure that
$$ \mathcal{E}(\cc(t)|\ww) \le e^{- \lambda\, t} \,  \mathcal{E}(\cc_0|\ww). $$
\medskip

Finally, we conclude the proof of Theorem \ref{e0} by using the
\medskip 

\begin{lemma}[Csisz\'{a}r-Kullback-Pinsker inequality, \cite{Arnold}]\label{CKPinequality}\hfill\\
	Assume that $\cc:\Omega \rightarrow \mathbb{R}_+$ is nonnegative and measurable and $\ww\in \mathbb{R}^N_{>0}$ such that $\mathbb Q\, \overline{\cc} = \mathbb Q\,\ww$. 
	
Then, there exists a constant $C_{CKP} >0$ depending only on $\Omega$ and $\mathbb Q$ such that
	\begin{equation*}
		C_{CKP}\sum_{i=1}^{N}\|c_i - c_{i,\infty}\|_{L^1(\Omega)}^2 \leq \mathcal{E}(\cc|\ww).
	\end{equation*}
\end{lemma}
\end{proof}


\subsection{A constructive entropy method and its application to a cycle of reactions}\label{sec:constructive}\hfill\\
Theorem \ref{EEDE} shows exponential convergence to equilibrium for the system \eqref{e0}, \eqref{Nin} when it is complex balanced and has no boundary equilibria. However, it does not give a quantitative estimate for the convergence rate $\lambda $. The reason is that the strictly positive limit in \eqref{h2}, which can be computed explicitly,
 is not necessary the positive infimum $\lambda_2$ in \eqref{lambda2}. Though this infimum comes out of an optimisation problem in finite dimension (and not from an abstract compactness argument in infinite dimension), it remains not explicitly computable.
 Therefore, in this subsection, we propose a constructive method to prove the entropy entropy-dissipation estimate \eqref{main} and to explicitly estimate the convergence rate $\lambda$.

This method builds on some recent ideas developed in \cite{FT15} for nonlinear detailed balance systems. It consists of four steps. The first three ones can be proven for any general complex balanced systems without boundary equilibria. The last step, being based on the structure of the mass conservation laws, is 
proven only for a specific network consisting of a cycle of reactions with arbitrary number of species, described in Thm. \ref{c_lem:eede}.
\medskip 

We begin therefore the 
\medskip

{\it{Proof of Thm. \ref{c_lem:eede}}} :
Since the three first steps of the proof hold for complex balance general systems without boundary equilibria, we use  the notations $\cc$, $\overline{\cc}$, $\ww$. For step four we change the notation to ${\bf{a}}$, ${\bf{\bar{a}}}$, ${\bf{a}}_{\infty}$ to underline that this step is specific to system 
\eqref{c_cycle_0}.

\medskip 

\noindent{\bf Step 1} (Additivity of the relative entropy and Logarithmic Sobolev Inequality)\\
As in the proof of Thm. \ref{EEDE}, we use the additivity property
		(for $\mathcal{E}$ given by formula \eqref{FreeEnergy_PDE}):
		\begin{equation*}	
			\mathcal{E}(\cc|\ww) = \mathcal{E}(\cc|\overline{\cc}) + \mathcal{E}(\overline{\cc}|\ww),
		\end{equation*}
		and control the term $\mathcal{E}(\cc|\overline{\cc})$ via the Logarithmic Sobolev Inequality as in \eqref{lam1} (with $\mathcal{D}$ defined by formula \eqref{ED_PDE}), i.e.
		\begin{equation*}
			\frac12 \mathcal{D}(\cc) \geq \lambda_1 \mathcal{E}(\cc|\overline{\cc}),
		\end{equation*}
		for an explicit constant $\lambda_1>0$. It remains therefore to control $\mathcal{E}(\overline{\cc}|\ww)$.
		\medskip
	
\noindent{\bf Step 2} (Upper and lower bounds using quadratic terms of square root concentrations)\\
		We estimate $\mathcal{D}(\cc)$ and $\mathcal{E}(\overline{\cc}|\ww)$ by quadratic terms involving
		 the square roots of concentrations, which are significantly easier to handle than logarithmic terms. 
Recalling the short hand notation $C_i := \sqrt{c_{i}}$ and $C_{i,\infty} := \sqrt{c_{i,\infty}}$ for square roots of concentrations, we estimate first
		\begin{equation}\label{sstep2}
			\frac12\mathcal{D}(\cc) \geq 2\min_{j}\{d_j\}\sum_{i=1}^{N} \|\nabla C_i\|^2  + \frac12\sum_{r=1}^{R}k_r\ww^{y_r}\biggl\|\frac{\CC^{y_r}}{\WW^{y_r}} - \frac{\CC^{y_r'}}{\WW^{y_r'}}\biggr\|^2,
		\end{equation}
thanks to inequality \eqref{elest}. Now, by recalling the conservation of mass \eqref{MassCons} and the nonnegativity of the solution, we see that $\overline{c_i}\leq M$ for all $i=1,\ldots, N$. Moreover, since $\Psi(\cdot, y)/(\sqrt{\cdot} - \sqrt{y})^2$ is increasing for each fixed $y>0$, the elementary estimate 
	\begin{equation*}
		\frac{\Psi(x;y)}{(\sqrt{x} - \sqrt{y})^2} \leq \frac{\Psi(M;y)}{(\sqrt{M} - \sqrt{y})^2}
		\end{equation*}
	 holds for all $x\in (0,M]$ and $y>0$,
and	we can estimate $\mathcal{E}(\overline{\cc}|\ww)$ as follows:
		\begin{equation}\label{E2quad}
			\mathcal{E}(\overline{\cc}|\ww) \leq \sum_{i=1}^{N}\frac{\Psi(M;c_{i,\infty})}{\bigl(\sqrt{M} - \sqrt{c_{i,\infty}}\bigr)^2}\left(\sqrt{\overline{C_i^2}} - C_{i,\infty}\right)^2 \leq K_2\sum_{i=1}^{N}\left(\sqrt{\overline{C_i^2}} - C_{i,\infty}\right)^2,
		\end{equation}
 where $K_2$ depends only on $M$ and $\ww$.
 		\medskip

\noindent{\bf Step 3} (Control of the reaction dissipation term via a reaction dissipation term for spatial averages)\\	
As another step towards exploiting the mass conservation laws, we estimate $\mathcal{D}(\cc)$ further by
		\begin{equation}\label{Dquad}
			\frac12\mathcal{D}(\cc) \geq K_1\biggl(\,\sum_{i=1}^{N}\|\nabla C_i\|^2 + \sum_{i=1}^{N}\biggl(\frac{\overline{\CC}^{y_r}}{\WW^{y_r}} - \frac{\overline{\CC}^{y_r'}}{\WW^{y_r'}}\biggr)^{\!2}\,\biggr)
		\end{equation}
		for an explicit constant $K_1>0$. This technical step follows from \cite[Lemma 2.6]{FT15} with slight modifications. 
			For the sake of completeness, we recall here the main arguments of the proof. \textcolor{black}{The proof makes use of a domain decomposition corresponding to the deviation to the averages of $C_i$, i.e. by denoting $\delta_i(x) = C_i(x) - \overline{C}_i$, we consider the decomposition 
			\begin{equation*}
				\Omega = S \cup S^{c}
			\end{equation*}
			where $S = \{x\in\Omega: |\delta_i(x)| \leq L \text{ for all } i = 1,\ldots, N \}$ with a fixed constant $L>0$ (which can be chosen arbitrarily). On the set $S$, thanks to a Taylor expansion, we have
			\begin{equation*}
				\frac{\mathbf{C}^{y_r}}{\mathbf{C}_{\infty}^{y_r}} = \frac{1}{\mathbf{C}_{\infty}^{y_r}}\prod_{i=1}^{N}\left(\overline{C_i} + \delta_i\right)^{y_{r,i}} = \frac{\overline{\mathbf{C}}^{y_r}}{\mathbf{C}_{\infty}^{y_r}} + \widetilde{R}(\overline{C_1},.., \overline{C_N}, \delta_1,.., \delta_N)\sum_{i=1}^{N}\delta_i
			\end{equation*}
			with $|\widetilde{R}(\overline{C_i}, \delta_i)| \leq C(L)$ thanks to the upper bounds {$\overline{C}_i\le \sqrt{\overline{c_i}} \le \sqrt{\tilde{K}}$} (obtained by a computation similar to (\ref{KK})) and $|\delta_i(x)| \leq L$ in $S$. It thus follows that
			\begin{equation}\label{z1}
				\frac12\sum_{r=1}^{R}k_r\ww^{y_r}\left\|\frac{\mathbf{C}^{y_r}}{\mathbf{C}_{\infty}^{y_r}} - \frac{\mathbf{C}^{y_r'}}{\mathbf{C}_{\infty}^{y_r'}}\right\|_{L^2(S)}^2 \geq  \beta_1\sum_{r=1}^{R}\left[\frac{\overline{\mathbf{C}}^{y_r}}{\mathbf{C}_{\infty}^{y_r}} - \frac{\overline{\mathbf{C}}^{y_r'}}{\mathbf{C}_{\infty}^{y_r'}}\right]^2|S| - \beta_2\sum_{i=1}^{N}\|\delta_i\|_{L^2(S)}^2
			\end{equation}
			for some constants $\beta_1, \beta_2>0$.
On the other hand, by using the lower bounds $|\delta_i|\ge L$ on $S^{c}$ for some $1\leq i\leq N$, the boundedness of $\overline{C_i}$, and Poincar\'e's inequality,  it follows that
			\begin{equation}\label{z2}
				\frac12\min_{j}\{d_j\}\sum_{i=1}^{N}\|\nabla C_i\|^2 \geq \beta_3\sum_{i=1}^{N}\|\delta_i\|^2_{L^2(\Omega)} \geq \beta_3 N L^2 |S^c| \geq  \beta_4\sum_{r=1}^{R}\left[\frac{\overline{\mathbf{C}}^{y_r}}{\mathbf{C}_{\infty}^{y_r}} - \frac{\overline{\mathbf{C}}^{y_r'}}{\mathbf{C}_{\infty}^{y_r'}}\right]^2|S^{c}|
			\end{equation}
			for constants $\beta_3, \beta_4 >0$. Note that all constants $\beta_1, \beta_2, \beta_3, \beta_4$ are independent of $S$.
			Now, a combination of \eqref{sstep2}, \eqref{z1} and \eqref{z2} leads to, for any $\theta \in (0,1)$,
			\begin{equation*}
				\begin{aligned}
					\frac{\mathcal{D}(\cc)}{2} &\geq \min_j\{d_j\}\sum_{i=1}^{N}\|\nabla C_i\|^2 + \!\left(\!\beta_3\!\sum_{i=1}^{N}\|\delta_i\|^2_{L^2(\Omega)} + \frac{\min_j\{d_j\}}{2}\sum_{i=1}^{N}\|\nabla C_i\|^2 + \frac{\theta}{2}\sum_{r=1}^{R}k_r\ww^{y_r}\left\|\frac{\mathbf{C}^{y_r}}{\mathbf{C}_{\infty}^{y_r}} - \frac{\mathbf{C}^{y_r'}}{\mathbf{C}_{\infty}^{y_r'}}\right\|^2\right)\\
					&\geq \min_j\{d_j\}\sum_{i=1}^{N}\|\nabla C_i\|^2 + \left(\beta_3 - \theta \beta_2\right)\sum_{i=1}^{N}\|\delta_i\|^2_{L^2(S)} + \min\{\theta\beta_1, \beta_4\}\sum_{r=1}^{R}\left[\frac{\overline{\mathbf{C}}^{y_r}}{\mathbf{C}_{\infty}^{y_r}} - \frac{\overline{\mathbf{C}}^{y_r'}}{\mathbf{C}_{\infty}^{y_r'}}\right]^2,
				\end{aligned}
			\end{equation*}
where we recall that $|S|+|S^{c}|=|\Omega|=1$.
			Hence, \eqref{Dquad} follows by choosing $L$, $0 < \theta < \min\{ \beta_3\beta_2^{-1}, 1\}$ and $K_1 = \min\{\theta\beta_1, \beta_4, \min_{j}\{d_j\}\}$.}
		\medskip
		
\noindent{\bf Step 4} (Conservation laws and deviations of averaged concentration around the equilibrium)\\
Since this step is specific to the cyclic reaction \eqref{c_cycle_0}, we change the notation to ${\mathbf{a}}$, $\overline{\mathbf{a}}$, ${\mathbf{a}}_{\infty}$, ${\mathbf{A}}$, etc. in the remainder of the proof. Moreover, we recall the notations
 $A_i = \sqrt{a_i}$, $A_{i,\infty} = \sqrt{a_{i,\infty}}$, and use the periodic notation $a_{N+1} := a_1$ or $a_{N+1,\infty} := a_{1,\infty}$, etc.
\medskip

\medskip 

First, we recall that  the system \eqref{c_cycle_0} is indeed a complex balanced system satisfying exactly 
one conservation law (corresponding to the conservation of the total mass): $\sum_{i=1}^N \frac{1}{\alpha_i}\int_{\Omega} a_i = M$. Indeed, the following
lemma holds:
\medskip 

	
	
	
	\begin{lemma}\label{c_equilibrium}
		For any strictly positive initial mass $M>0$, there exists a unique strictly positive complex balance equilibrium $\mathbf{a}_{\infty} = (a_{1,\infty}, a_{2,\infty}, \ldots, a_{N,\infty})$ to the system \eqref{c_cycle_0}, which solves
		\begin{equation}\label{sisc}
		\begin{cases}
		k_1a_{1,\infty}^{\alpha_1} = k_2a_{2,\infty}^{\alpha_2} = \ldots = k_Na_{N,\infty}^{\alpha_N},\\
		\frac{a_{1,\infty}}{\alpha_1} + \frac{a_{2,\infty}}{\alpha_2} + \ldots + \frac{a_{N,\infty}}{\alpha_N} = M.
		\end{cases}
		\end{equation}
	\end{lemma}
	\begin{proof}
	Using the first equation of \eqref{sisc},  we see that
	 $a_{i,\infty} = \left(\frac{k_1}{k_i}a_{1,\infty}^{\alpha_1}\right)^{1/\alpha_i}$ for all $i=2,3,\ldots, N$. Hence, $a_{1,\infty}$ solves
		\begin{equation*}
		\frac{a_{1,\infty}}{\alpha_1}  +\frac{1}{\alpha_2}\left(\frac{k_1}{k_2}a_{1,\infty}^{\alpha_1}\right)^{1/\alpha_2} + \ldots + \frac{1}{\alpha_N}\left(\frac{k_1}{k_N}a_{1,\infty}^{\alpha_1}\right)^{1/\alpha_N} = M.
		\end{equation*}
Denoting the left hand side of this equation by $f(a_{1,\infty})$, we see that $z \mapsto f(z)$ is a strictly increasing function on $[0,+\infty)$ with $f(0) = 0$ and $\lim_{z\rightarrow+\infty}f(z) = +\infty$. Thus, for any $M>0$, there exists a unique solution $a_{1,\infty}$ to the equation $f(z) = M$, which completes the proof of the lemma.
	\end{proof}
	
\medskip	
	
In the case of a cyclic reaction \eqref{c_cycle_0}, the relative entropy is specifically defined by \eqref{ee1},
while the corresponding entropy dissipation is given by \eqref{ee2}.

Moreover, following the above general method, 
it will be convenient for the readability of the proof to briefly restate the required key estimates of 
Steps 1-3:
First, we see that Step 1 becomes 
		\begin{equation}\label{c_e1}
			\mathcal{E}({\mathbf{a}}|{\mathbf{a}}_{\infty}) = \mathcal{E}({\mathbf{a}}|\overline{\mathbf{a}}) + \mathcal{E}(\overline{\mathbf{a}}|{\mathbf{a}}_{\infty}),
		\end{equation}	
		and, thanks to the logarithmic Sobolev inequality,
		\begin{equation}\label{c_e5}
			\frac12 \mathcal{D}({\mathbf{a}}) \geq \lambda_1\mathcal{E}({\mathbf{a}}|\overline{\mathbf{a}}).
		\end{equation}
		Step 2 shows that $\mathcal{D}({\mathbf{a}})$ and $\mathcal{E}(\overline{\mathbf{a}}|{\mathbf{a}}_{\infty})$ are controlled by quadratic terms of square roots of concentrations 
		\begin{align}\label{c_e6_0}
			\mathcal{D}({\mathbf{a}})&\geq 4\min_{j}\{d_j\}\sum_{i=1}^{N}\|\nabla A_i\|^2 + \sum_{i=1}^{N}k_{i}a_{i,\infty}^{\alpha_i}\biggl\|\frac{{A}_i^{\alpha_i}}{A_{i,\infty}^{\alpha_i}} - \frac{{A}_{i+1}^{\alpha_{i+1}}}{A_{{i+1},\infty}^{\alpha_{i+1}}}\biggr\|^2,\\
		\label{c_e7}
			\mathcal{E}(\overline{\mathbf{a}}|{\mathbf{a}}_{\infty})&\leq K_2\, \sum_{i=1}^N \left(\sqrt{\overline{A_i^2}} - A_{i,\infty}\right)^2.
		\end{align}
		In Step 3, we obtain moreover,
		\begin{equation}\label{c_e6}
			\mathcal{D}(\mathbf a)\geq K_1\biggl(\,\sum_{i=1}^{N}\|\nabla A_i\|^2 + \sum_{i=1}^{N}\biggl(\frac{\overline{A}_i^{\alpha_i}}{A_{i,\infty}^{\alpha_i}} - \frac{\overline{A}_{i+1}^{\alpha_{i+1}}}{A_{{i+1},\infty}^{\alpha_{i+1}}}\biggr)^2\,\biggr).
		\end{equation}
\medskip
		
In order to complete the proof of Thm. \ref{c_lem:eede} in the final Step 4, it therefore remains to show that
		\begin{equation}\label{remain}
			\sum_{i=1}^{N}\|\nabla A_i\|^2 + \sum_{i=1}^{N}\biggl(\frac{\overline{A}_i^{\alpha_i}}{A_{i,\infty}^{\alpha_i}} - \frac{\overline{A}_{i+1}^{\alpha_{i+1}}}{A_{{i+1},\infty}^{\alpha_{i+1}}}\biggr)^2
			\geq K_3 \sum_{i=1}^{N} \left(\sqrt{\overline{A_i^2}} - A_{i,\infty}\right)^2.
		\end{equation}
We recall the definition $\delta_i(x) = A_i(x) - \overline{A}_i$, which implies $\|\delta_i\|^2 = \overline{A_i^2} - \overline{A}_i^2$ for $i=1,\ldots, N$. By using the ansatz 
		\begin{equation}\label{c_ansatz}
		\overline{A_i^2} = A_{i,\infty}^2(1+\mu_i)^2 \quad \text{ with } \quad \mu_i\in [-1,+\infty)\qquad\text{for}\quad i=1,\ldots, N, 
		\end{equation}
		we compute
		\begin{equation*}
		\overline{A}_i = A_{i,\infty}(1+\mu_i) - \frac{\|\delta_i\|^2}{\sqrt{\overline{A_i^2}} + \overline{A}_i} =: A_{i,\infty}(1+\mu_i) - \|\delta_i\|^2R(A_i) \qquad\text{for all}\quad i=1,2,\ldots, N,
\end{equation*}
where we have defined $R(A_i):= \Bigl(\sqrt{\overline{A_i^2}} + \overline{A}_i\Bigr)^{-1}$.
Moreover, thanks to the mass conservation law $\sum_{i=1}^N \frac{1}{\alpha_i}\int_{\Omega} a_i(x)\, dx = M = \sum_{i=1}^{N}\frac{1}{\alpha_i}a_{i,\infty}$, we know that
		\begin{equation}\label{mass_cycle}
			\sum_{i=1}^{N}\frac{1}{\alpha_i}A_{i,\infty}^2\,\mu_i(\mu_i + 2) = 0.
		\end{equation}
Next, we can always fix an $\varepsilon>0$ small enough such that
		\begin{equation}\label{gamma}
			\gamma:= \min_{i,j\in\{1,..,N\}}\left\{\frac{1}{2^{\alpha_j+1}A_{j,\infty}^{2\alpha_j}}\left(\frac{M - \varepsilon^2/\alpha_{i}}{N - 1}\right)^{\alpha_j} - \frac{\varepsilon^{2\alpha_i}}{A_{i,\infty}^{2\alpha_i}}\right\} >0,
		\end{equation}
		and consider two cases.

\begin{description}[topsep=5pt, leftmargin=7mm]
\item[Case 4.1] If $\overline{A_i^2} \geq \varepsilon^2$ for all $i=1,\ldots, N$, then $R(A_i)\leq {\varepsilon}^{-1}$ for $i=1,2,\ldots, N$. 
Note that 
$$
	\mu_i = -1 + \frac{\sqrt{\overline{A_i^2}}}{A_{i,\infty}} \leq -1 + \frac{\sqrt{\alpha_i M}}{A_{i,\infty}}
$$
thanks to \eqref{c_ansatz} and $\|\delta_i\| \leq \sqrt{\overline{A_i^2}} \leq \sqrt{\alpha_i M}$.
Moreover, we use the ansatz \eqref{c_ansatz} and the Taylor expansion 
\begin{equation*}
	\frac{\overline{A}_i^{\alpha_i}}{A_{i,\infty}^{\alpha_i}} = \left((1+\mu_i) - \frac{\|\delta_i\|^2R(A_i)}{A_{i,\infty}}\right)^{\alpha_i}
	 = (1+\mu_i)^{\alpha_i} - \widetilde{R}_i\left(\mu_i, R(A_i), \|\delta_i\|, \alpha_i, A_{i,\infty}\right)\|\delta_i\| 
\end{equation*}
with a bounded remainder term $\widetilde{R}_i\left(\mu_i, R(A_i), \|\delta_i\|, \alpha_i, A_{i,\infty}\right)\leq C(\varepsilon)$ thanks to the bounds on $\mu_i$, $R(A_i)$ and $\|\delta_i\|$.
This yields for a $\theta\in (0,1)$ (to be chosen) the lower bound  
			\begin{equation}\label{t1}
				\begin{aligned}
				\text{LHS of \eqref{remain}}
					&\geq  \sum_{i=1}^{N}\|\nabla A_i\|^2 + 
					\theta\sum_{i=1}^{N}\biggl[\left((1+\mu_i)^{\alpha_i} - \widetilde{R}_i\|\delta_i\|\right)\\
					&\hspace{5cm} -  \left((1+\mu_{i+1})^{\alpha_{i+1}} - \widetilde{R}_{i+1}\|\delta_{i+1}\|\right)\biggr]^2\\
					&\geq \sum_{i=1}^{N}\|\nabla A_i\|^2 + \frac{\theta}{2} \sum_{i=1}^{N}\bigg((1+\mu_i)^{\alpha_i} - (1 +\mu_{i+1})^{\alpha_{i+1}}\bigg)^2 - \theta C(\varepsilon)\sum_{i=1}^{N}\|\delta_i\|^2\\
					&\geq \frac{\theta}{2}\sum_{i=1}^{N}\bigg((1+\mu_i)^{\alpha_i} - (1+\mu_{i+1})^{\alpha_{i+1}}\bigg)^2 ,
				\end{aligned}
			\end{equation}
where we have applied the Poincar\'{e}'s inequality and chosen $\theta$ small enough in the last step.
Moreover, the right hand side of \eqref{remain} is bounded above by
			\begin{equation}\label{t2}
			K_3 \max_{i=1,\ldots, N}\{A_{i,\infty}^2\}\sum_{i=1}^{N}\mu_i^2. 
			\end{equation} 
			Thanks to Lemma \ref{finite_cycle} below, we obtain \eqref{remain} from \eqref{t1} and \eqref{t2} by choosing 
			\begin{equation}\label{K3_1}
				K_3 \leq \frac{\theta}{2N^N\max_{i=1,\ldots,N}\{A_{i,\infty}^2\}}.
			\end{equation}
			
	\begin{lemma}\label{finite_cycle}
		Assume that $\mu_i \in [-1,+\infty)$ and $\alpha_i \in \mathbb N_{>0}$, $i=1,2,\ldots, N$ satisfy (for some $A_{i,\infty} \ge 0$),
		\begin{equation}\label{mass_cycle_1}
			\sum_{i=1}^{N}\frac{1}{\alpha_i}A_{i,\infty}^2(\mu_i^2 + 2\mu_i) = 0.
		\end{equation}
		Then, 
		\begin{equation}\label{c_e2_1}
			\sum_{i=1}^{N}\Bigl((1+\mu_i)^{\alpha_i} - (1+\mu_{i+1})^{\alpha_{i+1}}\Bigr)^2 \geq \frac{1}{N^{N}}\sum_{i=1}^{N}\mu_i^2.
		\end{equation}
	\end{lemma}
	\begin{proof}
					Note that the left hand side of \eqref{c_e2_1} can be bounded below by
					\begin{equation}\label{c_e3}
					\sum_{i=1}^{N}\bigg((1+\mu_i)^{\alpha_i} - (1+\mu_{i+1})^{\alpha_{i+1}}\bigg)^2 \geq \frac{1}{N^N} \sum_{i\not=j}\biggl((1+\mu_i)^{\alpha_i} - (1+\mu_j)^{\alpha_j}\biggr)^2,
					\end{equation}
thanks to the elementary inequality $a^2 + b^2 \geq \frac{1}{2}(a+b)^2$. Thanks to the mass conservation \eqref{mass_cycle_1}, there exists $\emptyset \not= I_0\subset \{1,2,\ldots, N\}$ such that $\mu_i \geq 0$ for $i\in I_0$ and $\mu_j <0$ for $j\in I_0^C$, where also the complement $I_0^C$ is not empty. Thus, for $i\in I_0$ and $j\in I_0^C$, we have
\begin{equation*}
	(1+\mu_i)^{\alpha_i} - (1+\mu_j)^{\alpha_j} \geq (1+\mu_i) - (1+\mu_j) = \mu_i - \mu_j \geq 0.
\end{equation*}
Then
\begin{equation*}
\Bigl((1+\mu_i)^{\alpha_i} - (1+\mu_j)^{\alpha_j}\Bigr)^2 \geq (\mu_i - \mu_j)^2 \geq \mu_i^2 + \mu_j^2
	\end{equation*}
	for all $i\in I_0$ and $j\in I_0^C$. We can thus continue to estimate \eqref{c_e3} 
\begin{equation}\label{c_e4}
\begin{aligned}
\sum_{i=1}^{N}\Bigl((1+\mu_i)^{\alpha_i} - (1-\mu_{i+1})^{\alpha_{i+1}}\Bigr)^2
&\geq \frac{1}{N^N}\sum_{i\not=j}\Bigl((1+\mu_i)^{\alpha_i} - (1+\mu_j)^{\alpha_j}\Bigr)^2\\
&\geq \frac{1}{N^N}\sum_{i\in I_0, j\in I_0^C}(\mu_i^2 + \mu_j^2)
\geq \frac{1}{N^{N}}\sum_{i=1}^{N}\mu_i^2.
\end{aligned}
\end{equation}
This ends the proof of Lemma \ref{finite_cycle} and Case 4.1 
of Thm. \ref{c_lem:eede}.
\end{proof}

\begin{remark} The proof of the above Lemma \ref{finite_cycle} used the conservation law \eqref{mass_cycle_1} in a rather weak way. Indeed, it only 
uses the fact that not all $\mu_i$ have the same sign.
We believe that
the ideas used in the proof of Lemma \ref{finite_cycle}  can be generalised
 to many complex balanced systems, 
once the structure of the conservation laws is explicitly given.
\end{remark}

\item[Case 4.2] We now consider cases where there exists (at least) one $i_0$ such that $\overline{A_{i_0}^2} \leq \varepsilon^2$. For sufficiently small $\varepsilon$, such cases can only occur away from the unique strictly positive equilibrium state. Hence, we expect that the entropy dissipation (and its lower bound on the left-hand-side of  \eqref{remain}) should be bounded below by a positive constant, which expresses this distance from the equilibrium state.	 

At first, however, we note that the right hand side of \eqref{remain} (which is an upper bound of the relative entropy towards the equilibrium) is bounded above by a constant due to the following estimate:
\begin{equation}\label{boundrhs}
	\hspace{5mm} K_3\sum_{i=1}^{N}\left(\sqrt{\overline{A_i^2}} - A_{i,\infty}\right)^2 \leq 2K_3\sum_{i=1}^{N}(\overline{a}_i + a_{i,\infty}) \leq 2K_3\biggl(M\max\{\alpha_i\} + \sum_{i=1}^{N}a_{i,\infty}\biggr).
\end{equation}
Hence, in order to prove \eqref{remain}, it remains  to show that the left hand side of \eqref{remain} is bounded below by a positive constant. To show that, we 
define
\begin{equation}\label{eta}
	\eta := \min_{i}\frac{M-\varepsilon^2/\alpha_i}{2(N-1)},
	\end{equation}
and distinguish two further subcases representing states with large and small spatial variations: 
\medskip 
			
\begin{description}[topsep=5pt, leftmargin=5mm]
\item[Subcase 1] There exists   ${i_*} \in \{1,2,\ldots, N\}$ such that $\|\delta_{i_*}\|^2 \geq \eta>0$. Then, we can estimate directly thanks to \eqref{boundrhs} and Poincar\'{e}-Wirtinger's inequality:
\begin{equation}\label{boundD1}
					\text{LHS of (\ref{remain})} \geq  \sum_{i=1}^{N}C_P\|\delta_{i}\|^2 \geq C_P\eta \geq K_3\sum_{i=1}^{N}\left(\sqrt{\overline{A_i^2}} - A_{i,\infty}\right)^2 ,
				\end{equation}
	by choosing
				\begin{equation}\label{K3_2}
					K_3 \leq \frac{C_P\eta}{2\left(M\max\{\alpha_i\} + \sum_{i=1}^{N}c_{i,\infty}\right)}.
				\end{equation}
\item[Subcase 2] For all  $i=1,2,\ldots, N$,  we have $\|\delta_i\|^2 \leq \eta$. In this case concerning states away from equilibrium yet with small spatial variation, we expect that the reaction terms represent in large parts the distance from the equilibrium.
This can be quantified from firstly recalling the conservation law 
$
\sum_{i=1}^{N}\frac{1}{\alpha_i}\overline{A_i^2} = M$, 
and from observing the estimate	
\begin{equation*}
\sum_{i\not=i_0}\frac{1}{\alpha_i}\overline{A_i^2} = M - \frac{1}{\alpha_{i_0}}\overline{A_{i_0}^2} \geq M - \varepsilon^2/\alpha_{i_0}.
\end{equation*}
	Thus, there exists $j_0 \not= i_0$ such that 
				\begin{equation*}
				\overline{A_{j_0}^2} \geq \alpha_{j_0}\frac{M - \varepsilon^2/\alpha_{i_0}}{N - 1} \geq \frac{M - \varepsilon^2/\alpha_{i_0}}{N - 1}
				\end{equation*}
				since $\alpha_{j_0} \geq 1$. 				Hence, by using \eqref{eta},
				\begin{equation*}
				\overline{A}_{j_0}^2 = \overline{A_{j_0}^2} - \|\delta_{j_0}\|^2 \geq \frac{M - \varepsilon^2/\alpha_{i_0}}{N - 1} - \eta \geq  \frac{M - \varepsilon^2/\alpha_{i_0}}{2(N - 1)}.
				\end{equation*}
Now, with $\gamma$ defined in \eqref{gamma}, by using the triangle inequality and Young's inequality, we get
\begin{align}
\text{LHS of \eqref{remain}} &\geq  \sum_{i=1}^{N}\left(\frac{\overline{A}_i^{\alpha_i}}{A_{i,\infty}^{\alpha_i}} - \frac{\overline{A}_{i+1}^{\alpha_{i+1}}}{A_{i+1,\infty}^{\alpha_{i+1}}}\right)^2 \geq \frac{1}{|j_0 - i_0| + 1}\left(\frac{\overline{A}_{i_0}^{\alpha_{i_0}}}{A_{i_0,\infty}^{\alpha_{i_0}}} - \frac{\overline{A}_{j_0}^{\alpha_{j_0}}}{A_{j_0,\infty}^{\alpha_{j_0}}}\right)^2\nonumber\\
&\geq \frac{1}{N}\left(\frac{1}{2}\frac{\overline{A}_{j_0}^{2\alpha_{j_0}}}{A_{j_0,\infty}^{2\alpha_{j_0}}} - \frac{\overline{A}_{i_0}^{2\alpha_{i_0}}}{A_{i_0,\infty}^{2\alpha_{i_0}}}\right)\nonumber
\geq \frac{1}{N}\left(\frac{1}{2^{\alpha_{j_0}+1}A_{j_0,\infty}^{2\alpha_{j_0}}}\left(\frac{M-\varepsilon^2/\alpha_{i_0}}{N-1}\right)^{\alpha_{j_0}} - \frac{\varepsilon^{2\alpha_{i_0}}}{A_{i_0,\infty}^{2\alpha_{i_0}}}\right)\\
&\geq \frac{\gamma}{N}
\geq K_3\sum_{i=1}^{N}\left(\sqrt{\overline{A_i^2}} - A_{i,\infty}\right)^2, \label{boundD2}
\end{align}
whenever
\begin{equation}\label{K3_3}
	K_3 \leq \frac{\gamma}{2N\left(M\max\{\alpha_i\} + \sum_{i=1}^{N}c_{i,\infty}\right)}.
\end{equation}
\end{description}
\end{description}

From \eqref{K3_1}, \eqref{K3_2} and \eqref{K3_3}, we obtain \eqref{remain} with an explicit $K_3$.
Then by combining \eqref{c_e1}, \eqref{c_e5} and \eqref{remain},
we can conclude the desired entropy entropy-dissipation inequality. 

The second part of the proof of Thm. \ref{c_lem:eede}
concerns the exponential convergence of the renormalised solution towards the strictly positive complex balance equilibrium. 

Firstly, if the solution of the system \eqref{c_cycle_0} is regular enough in order to 
 rigorously satisfy the 
weak entropy entropy-dissipation law \eqref{integrability}, i.e., for a.a. $t\geq s > 0$,
\begin{equation*}
	\mathcal{E}(\mathbf a|\mathbf a_{\infty})(t) + \int_{s}^{t}\mathcal{D}(\mathbf a)(r)dr \leq \mathcal E(\mathbf a|\mathbf a_{\infty})(s),
\end{equation*}
then exponential convergence to equilibrium in relative entropy $\mathcal{E}(\mathbf a|\mathbf a_{\infty})$
follows from a suitable Gronwall argument, see e.g. \cite{Wil} or more specifically \cite{FL15}.
Moreover, the Csisz\'{a}r-Kullback-Pinsker inequality yields exponential convergence of solutions to the complex balance equilibrium in $L^1$.
\medskip 

However, for renormalised solution as introduced in \cite{Fis15}, it is unclear if the weak entropy entropy-dissipation law  \eqref{integrability} holds due to the lacking integrability of the nonlinear reaction terms in the entropy-dissipation.
We will resolve this issue by proving exponential convergence to equilibrium with a uniform rate for the family of  
approximating sequences which were used to construct renormalised solutions in \cite{Fis15}. 

Since the below estimates actually hold for general systems and not just for the cyclic reaction  \eqref{c_cycle_0},
we revert to the notations for general complex balanced systems \eqref{e0}, i.e.  
\begin{equation*}
	\partial_t\cc - \mathbb D\Delta \cc = \mathbf R(\cc), \quad \nabla \cc \cdot \nu = 0, \quad \cc(x,0) = \cc_0(x).
\end{equation*}
The corresponding approximating systems used by \cite{Fis15} write as 
\begin{equation}\label{approx}
	\partial_t\cc^{\varepsilon} - \mathbb D\Delta \cc^{\varepsilon} = \frac{\mathbf R(\cc^{\varepsilon})}{1+\varepsilon|\mathbf R(\cc^{\varepsilon})|}, \quad \nabla \cc^{\varepsilon} \cdot \nu = 0, \quad \cc^{\varepsilon}(x,0) = \cc^{\varepsilon}_0(x),
\end{equation}
where $\cc^{\varepsilon}_0$ is a suitable approximation of $\cc_0$ as $\varepsilon\rightarrow 0$.
It was proven in \cite{Fis15} that \eqref{approx} has a unique weak solution $\cc^{\varepsilon}$,
 with $\cc^{\varepsilon} \in L^{\infty}([0,+\infty[; L\log L(\Omega))$ and, as $\varepsilon \rightarrow 0$,
\begin{equation}\label{a.e.}
	\cc^{\varepsilon} \rightarrow \cc \quad \text{ almost everywhere on } \Omega\times [0,+\infty),
\end{equation}
where $\cc$ is the renormalised solution of \eqref{e0}. Note that \eqref{e0} and \eqref{approx} share the same complex balance equilibrium $\ww$. Moreover, we use the same relative entropy for \eqref{approx} as
for the limiting system: 
\begin{equation*}
	\mathcal E(\cc^{\varepsilon}|\ww) = \sum_{i=1}^{N}\int_{\Omega}\left(c_i^{\varepsilon}\log{\frac{c_i^{\varepsilon}}{c_{i,\infty}}} - c_i^{\varepsilon} + c_{i,\infty}\right)dx.
\end{equation*}
By using Proposition \ref{cale} and the notation $G(\cc)$ in \eqref{322bis}, we compute the entropy-dissipation for \eqref{approx} as
\begin{equation*}
	\mathcal{D}^{\varepsilon}(\cc^{\varepsilon}) = -\frac{d}{dt}\mathcal{E}(\cc^{\varepsilon}|\ww) = \sum_{i=1}^{N}d_i\int_{\Omega}\frac{|\nabla c^{\varepsilon}_i|^2}{c^{\varepsilon}_i}dx + \int_{\Omega}\frac{G(\cc^{\varepsilon})}{1+\varepsilon|\mathbf R(\cc^{\varepsilon})|}dx.
\end{equation*}
We can now follow the first part of the proof of Thm. \ref{c_lem:eede} to show that
\begin{equation}\label{desired}
	\mathcal{D}^{\varepsilon}(\cc^{\varepsilon}) \geq \eta\, \mathcal{E}(\cc^{\varepsilon}|\ww)
\end{equation}
where {\it $\eta>0$ is independent of $\varepsilon$}. Indeed, in the first step, we write
\begin{equation*}
	\mathcal{E}(\cc^{\varepsilon}|\ww) = \mathcal E(\cc^{\varepsilon}|\overline{\cc^{\varepsilon}}) + \mathcal E(\overline{\cc^{\varepsilon}}|\ww),
\end{equation*}
in which the first term is controlled by $\mathcal D^{\varepsilon}(\cc^{\varepsilon})$ thanks to the Logarithmic Sobolev inequality
\begin{equation*}
	\frac{1}{2}\mathcal D^{\varepsilon}(\cc^{\varepsilon}) \geq \lambda_1\mathcal E(\cc^{\varepsilon}|\overline{\cc^{\varepsilon}}),
\end{equation*}
where $\lambda_1$ is independent of $\varepsilon$. In the second step, we directly estimate $\mathcal{D}^{\varepsilon}(\cc^{\varepsilon})$ below by
\begin{equation*}
	\mathcal D^{\varepsilon}(\cc^{\varepsilon}) \geq 4\min_{j}\{d_j\}\sum_{i=1}^{N}\|\nabla C^{\varepsilon}_i\|^2 + \int_{\Omega}\frac{H(\mathbf C^{\varepsilon})}{1+\varepsilon|\mathbf R (\cc^{\varepsilon})|}dx,
\end{equation*}
where
\begin{equation*}
	H(\mathbf C^{\varepsilon}) = \sum_{r=1}^{|\mathcal R|}k_r\ww^{y_r}\left(\frac{(\CC^{\varepsilon})^{y_r}}{\WW^{y_r}} - \frac{(\CC^{\varepsilon})^{y_r'}}{\WW^{y_r'}}\right)^2.
\end{equation*}
Following Step 3 and  using the domain decomposition $\Omega=S\cup S^{c}$, we observe first that on $S = \{x\in\Omega: |\delta_i(x)| \leq L \text{ for all } i = 1,\ldots, N \}$. 
{Secondly, we note that for cyclic reactions satisfying the mass conservation law \eqref{MassCons}, we estimate 
 $\overline{C_i^{\varepsilon}}^2\le \overline{c_i^{\varepsilon}}\leq \alpha_i M$ thanks to Jensen's inequality. 
Therefore, we have 
$|C_i^{\varepsilon}|\le |\delta_i|+ \overline{C_i^{\varepsilon}}$ and thus
$|\mathbf R (\cc^{\varepsilon})| \le C(L,M)$ for some constant $C(L,M)$ depending
on $L$ and $M$ (and the rates appearing in $\mathbf{R}$).  
Note that for general systems, we can equally apply the upper bounds \eqref{KK} instead of \eqref{MassCons} and replace $M$ by $\tilde K$.}
 
Hence, we obtain for $0\le \varepsilon \le 1$ the uniform lower bound  
$$
\frac{1}{1+\varepsilon|\mathbf R (\cc^{\varepsilon})|}\ge \frac{1}{1+C(L,M)}.
$$
As a consequence, with this simple modification, we can show as in the proof of Thm. \ref{c_lem:eede} that 
\begin{equation*}
	\mathcal D^{\varepsilon}(\cc^{\varepsilon}) \geq K_1\left(\sum_{i=1}^{N}\|\nabla C^{\varepsilon}_i\|^2 + \sum_{i=1}^{N}\left(\frac{\overline{\CC^{\varepsilon}}^{y_r}}{\WW^{y_r}} - \frac{\overline{\CC^{\varepsilon}}^{y_r'}}{\WW^{y_r'}}\right)^2\right)
\end{equation*}
where {\it $K_1>0$ is independent of $\varepsilon$}. Note that the statement involving  $S^{c}$ holds without change. 

Moreover, we can now repeat all the arguments in Step 4 to finally prove estimate \eqref{desired}. Hence, by applying the Gronwall lemma, 
we obtain convergence to equilibrium in relative entropy, i.e.
\begin{equation}\label{desired_1}
	\mathcal E(\cc^{\varepsilon}(t)|\ww) \leq e^{-\eta t}\mathcal{E}(\cc^{\varepsilon}_0|\ww),
\end{equation}
for all $t>0$ and with a rate $\eta$ which is independent of $\varepsilon$. Thanks to the almost everywhere convergence \eqref{a.e.} and the convexity of $\mathcal E$, we can pass to the limit $\varepsilon \rightarrow 0$ in \eqref{desired_1}, and end up with
\begin{equation*}
	\mathcal{E}(\cc(t)|\ww) \leq \liminf_{\varepsilon\rightarrow 0}\mathcal{E}(\cc^{\varepsilon}(t)|\ww) \leq e^{-\eta t}\mathcal{E}(\cc_0|\ww),
\end{equation*}
which, combined with the Csisz\'{a}r-Kullback-Pinsker inequality, allows us to finish the proof of Thm. \ref{c_lem:eede}.\hfill  $\blacksquare$

\section{Complex balanced systems with boundary equilibria}\label{sec:boundary}

In the previous section, we have shown that for a complex balanced reaction-diffusion system without boundary equilibria,  all solution trajectories converge exponentially fast 
to the unique strictly positive complex balance equilibrium thanks to a functional entropy entropy-dissipation inequality.

It was also pointed out in the introduction that if a system has boundary equilibria, then such an entropy entropy-dissipation estimate 
cannot hold with the same generality. 
Indeed, if a trajectory approaches a boundary equilibrium, then the entropy dissipation tends to zero while the relative 
entropy functional towards the positive complex balance equilibrium remains strictly positive, see \eqref{compute}. The question of the large time behaviour for general complex balanced reaction-diffusion systems possessing boundary equilibria is
therefore an open problem. 
\medskip

In this section, we give  some partial answers. 
We first consider 
 the two by two model \eqref{2x2_0},
 and then
  the three by three model \eqref{3x3_0}.
 In the two by two model, the boundedness of solutions away from the boundary equilibrium $(0, M)$ follows from a comparison-principle argument. The instability of the boundary equilibrium $(0,0,M)$ for the three by three model, however, turns out be tricky and makes system \eqref{3x3_0} a more interesting example. 
 
In fact, we do not obtain the instability directly. To establish Thm. \ref{ExpConverge}, we instead prove that if a trajectory should approach the boundary equilibrium, then the rate of this convergence process cannot be faster than $1/(1+t)$. That (too) slow convergence is sufficient to apply an entropy entropy-dissipation-like estimate which yields instead convergence to the unique strictly positive complex balance equilibrium with an algebraic rate. 

In a second step, thanks to this algebraically fast convergence, it follows after some positive time $T>0$ that such trajectories remain outside of a neighbourhood of the boundary equilibrium. 
Hence, another entropy entropy-dissipation estimate can be proven, which implies exponential convergence to the unique strictly positive complex balance equilibrium.


\begin{remark}\label{generalisation}
It is easy to see that the result of Theorem \ref{ExpConverge} can be
 extended
  to the following class of complex balanced reaction networks: 
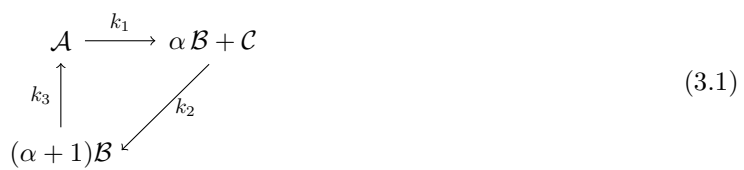
\begin{equation} \label{gen}
\begin{tikzpicture}[baseline=(current  bounding  box.center)]
	\node (a) {$\mathcal A$} node (b) at (2,0) {$\alpha \,\mathcal B+\mathcal C$} node (c) at (0,-1.5) {$(\alpha + 1)\mathcal B$};
	\draw[arrows=->] ([xshift =0.5mm]a.east) -- node [above] {\scalebox{.8}[.8]{$k_1$}} ([xshift =-0.5mm]b.west);
	\draw[arrows=->] ([yshift=0.5mm]c.north) -- node [left] {\scalebox{.8}[.8]{$k_3$}} ([yshift=-0.5mm]a.south);
	\draw[arrows=->] ([xshift =-0.5mm,yshift=-0.5mm]b.south) -- node [right] {\scalebox{.8}[.8]{$k_2$}} ([yshift =0.5mm]c.east);
\end{tikzpicture}
\end{equation}
with $\alpha \in \mathbb{N}_{>0}$. All these systems  possess the boundary equilibrium $(0,0,M)$. 


\end{remark}
\medskip


\textcolor{black}{It is important to notice a substantial difference between the ODE and the PDE settings of \eqref{2x2_0} and \eqref{3x3_0}. In the ODE setting it can be easily verified that any trajectory starting away from the boundary equilibria (even when starting at the boundary $\partial\mathbb R^2_+$ or $\partial\mathbb R^3_+$, respectively), will enter the interior of the positive quadrant and thus eventually converge to the unique strictly positive complex balance equilibrium. The PDE setting, however, allows the existence of a non-trivial class of initial data, for which solutions converge to the boundary equilibrium. We present such an example for system \eqref{3x3_0} in Rmk. \ref{r44}.}
\medskip

We now start the 
\medskip

\begin{proof}[Proof of Prop.\ref{2x2_theorem}]
We show first the existence of a unique positive equilibrium and a boundary equilibrium for \eqref{2x2_0}. Indeed, the system \eqref{2x2_0} satisfies the mass conservation
\begin{equation*}
	\int_{\Omega}(a(x,t) + b(x,t))dx = M:= \int_{\Omega}(a_0(x) + b_0(x))dx \qquad \forall\, t>0.
\end{equation*}
Thus, an equilibrium of \eqref{2x2_0} solves
\begin{equation*}
	\begin{cases}
		a_{\infty}^2 = a_{\infty}b_{\infty},\\
		a_{\infty} + b_{\infty} = M,
	\end{cases}
\end{equation*}
so that there exists a unique strictly positive equilibrium $(a_{\infty}, b_{\infty}) = (M/2, M/2)$,
 and there also exists a boundary equilibrium $(a^*, b^*) = (0,M)$.

Next, we show the uniform propagation of lower and upper bounds for solutions to \eqref{2x2_0}. While comparison principles cannot be expected to hold for general systems of parabolic equations, the specific structure of the $2\times2$ system \eqref{2x2_0} 
allows to prove the following {\it{a priori}} bounds and global existence of classical solutions:

\begin{lemma}[Uniform propagation of positive lower and upper bounds and existence of global classical solutions]\label{bound2x2}\hfill\\
Under the assumption of Prop. \ref{2x2_theorem},
there exists a unique global classical solution $(a, b)$ to \eqref{2x2_0} satisfying the same bounds
	\begin{equation} \label{bbo}
	\varepsilon^2 \leq a(t,x), b(t,x) \leq \Lambda \qquad \text{ for all }\quad x\in\Omega, \quad t>0.
	\end{equation}
\end{lemma}
\begin{proof}
	The propagation of the claimed strictly positive lower and upper bounds to solutions of system \eqref{bound2x2} follows, for instance, from comparison principle arguments. Following e.g. \cite{kirane}, one can argue that if one of the (possibly regularised) solutions $a$ or $b$ should violate these positive lower or upper bounds at some position $x_0\in\Omega$ for a first time at $t_0>0$, then the right hand side of the corresponding equations (for the concentration minus the lower/upper bound) 
	has the proper sign to yield a contradiction with the parabolic minimum/maximum principle for such extremal points. 
	
	Alternatively, the following argument (which works for the same structural reasons as the comparison principle argument) equally allows to prove the propagation of the positive lower and upper bounds as limits of a hierarchy 
	of corresponding $L^p$-bounds. 
We formally compute
	(the formal computation can easily be made rigorous via approximating systems with strictly positive solutions e.g. by adding vanishing positive source terms) 
	the quantity leading to the lower bound (the upper bound can be proven in a similar way):
	\begin{multline*}
	\frac{d}{dt}\int_{\Omega}\left(\frac{1}{pa^p} + \frac{1}{pb^p}\right)dx\\
	=-\int_{\Omega}d_a(p+1)a^{-p-2}|\nabla a|^2dx - \int_{\Omega}d_b(p+1)b^{-p-2}|\nabla b|^2dx
	+ \int_{\Omega}a(a^{-p-1} - b^{-p-1})(a-b)dx\leq 0.
	\end{multline*}
	Thus,
	\begin{equation*}	
	\sup_{t>0}\left\{\int_{\Omega}\Bigl(\frac{1}{a^p} + \frac{1}{b^p}\Bigr)dx\right\}  \leq \int_{\Omega}\Bigl(\frac{1}{a_0^p} + \frac{1}{b_0^p}\Bigr)dx \ 
	\xrightarrow{p\to\infty}\
	\sup_{t\ge 0}\left\{\left\|\frac{1}{a(t)}\right\|_{L^{\infty}_x}+ \left\|\frac{1}{b(t)}\right\|_{L^{\infty}_x}\right\} \leq \left\|\frac{1}{a_0}\right\|_{L^{\infty}_x} + \left\|\frac{1}{b_0}\right\|_{L^{\infty}_x}.
	\end{equation*}
	Finally, due to the uniform-in-time $L^{\infty}$-bounds \eqref{bbo}, the existence of a unique global classical solution 
	to \eqref{2x2_0} (which satisfies the positive lower and upper bounds point-wise in $x$) follows from classical arguments.
\end{proof}	

Lemma \ref{bound2x2} implies that initial data $(a_0, b_0)$, which are a.e. bounded away from the boundary equilibrium $(0,M)$, yield solutions $(a(\cdot,t), b(\cdot,t))$, which are uniformly-in-time bounded away from the same boundary equilibrium.  

	Then, direct computations show that
	\begin{equation}\label{l1}
	\begin{aligned}
	\frac{1}{2}\frac{d}{dt}\left(\|a(\cdot,t) - a_{\infty}\|^2 + \|b(\cdot,t) - b_{\infty}\|^2\right)
	&= - d_a\int_{\Omega}|\nabla a|^2dx - d_b\int_{\Omega}|\nabla b|^2dx - \int_{\Omega}a(a - b)^2dx\\
	&\leq - d_a\int_{\Omega}|\nabla a|^2dx - d_b\int_{\Omega}|\nabla b|^2dx - \varepsilon^2\int_{\Omega}(a - b)^2dx ,
	\end{aligned}
	\end{equation}
	thanks to Lemma \ref{bound2x2}. Due to the mass conservation $\overline{a}(t) + \overline{b}(t) = M$ and previously established entropy entropy-dissipation estimates for detailed balance reaction-diffusion systems without boundary equilibria (see e.g. \cite{DeFe06,DeFe08,FT15}), we obtain, for all $t>0$,  
	\begin{equation*}
	d_a\int_{\Omega}|\nabla a|^2dx + d_b\int_{\Omega}|\nabla b|^2dx + \varepsilon^2\int_{\Omega}(a - b)^2dx
	\geq \frac{\lambda_{\varepsilon}}{2}\left(\|a(\cdot,t) - a_{\infty}\|^2 + \|b(\cdot,t) - b_{\infty}\|^2\right),
	\end{equation*}
	where $\lambda_{\varepsilon}$ depends only on $\Omega$, $M$, $d_a, d_b$ (and $\varepsilon$) and moreover $\lambda_{\varepsilon} = O(\varepsilon^2)$. By inserting the above estimate into \eqref{l1}, exponential convergence to equilibrium follows from a standard Gronwall inequality.
\end{proof}

\begin{proof}[Proof of Cor. \ref{2x2cor}]
First, we remark that classical existence results imply for each continuous  nonnegative initial data a unique global classical nonnegative (in the sense that both components are nonnegative) solution $(a,b)$ to the system \eqref{2x2_0}. Moreover, these solutions satisfy an $L^{\infty}$ upper bound
as in the proof of Prop. \ref{2x2_theorem}. 
Thus, we get the estimate
$$
\begin{cases}
a_t - d_a  \Delta a \ge -\|a\|_{L^{\infty}} a, \\
a_0(x)\ge 0 \quad \text{with} \quad \int_{\Omega} a_0(x)\,dx >0.
\end{cases}
$$
Hence, the strong minimum principle for the heat-equation with Neumann boundary conditions (see e.g. \cite{Pao}) 
implies for classical solutions that for any $\tau>0$ there exists an $\varepsilon_a(\tau)>0$ such that $a(x,\tau)\ge \varepsilon_a(\tau)$. Otherwise, any point 
where $a(x,\tau)=0$ (with $x \in \Omega$) would be a minimum and yield a contradiction with the minimum principle except if $a_0(x)\equiv 0$ on $\Omega$, which we have excluded of the set of admissible initial data (note 
that for homogeneous Dirichlet data, we would indeed only have $a(x,\tau)> 0$ on $\Omega$). 
The same argument holds for $b(x,t)$ provided that 
$\int_{\Omega} b_0(x)\,dx >0$. In the special case when
$\int_{\Omega} b_0(x)\,dx =0$, we observe 
 that 
$$
\frac{d}{dt} \int_{\Omega} b(x,t)\,dx \bigg|_{t=0}  \geq \int_{\Omega} a_0^2(x)\,dx   
>0.
$$
Hence, for any (sufficiently small) $\tau/2>0$, we see that $\int_{\Omega} b(x,\tau/2)\,dx>0$ 
and, via an analog minimum principle argument, that $b(x,\tau)\ge \varepsilon_b(\tau)$ for some $\varepsilon_b(\tau)>0$.
As a consequence, the statement of Corollary \ref{2x2cor} follows from applying Proposition \ref{2x2_theorem} 
for $t\ge\tau>0$.

For the second statement of Cor. \ref{2x2cor}, we observe that $a_0\equiv0$ and  \eqref{2x2_0} imply that the solution $a(x,t)\equiv0$ for all times, and thus 
$$
\begin{cases}
b_t - d_b\Delta b = 0, &x\in\Omega, \quad t>0,\\
b_0(x)\ge 0, &x\in\Omega, \quad \int_{\Omega} b_0(x)\,dx=M.
\end{cases}
$$
Therefore, by standard results for the heat equation and for all $t>0$, we have 
\begin{equation*}
\|b(t) - M\|^2 \leq e^{-d_b\lambda_P t}\|b_0 - M\|^2,
\end{equation*}
where $\lambda_P$ is the best constant in Poincar\'{e}-Wirtinger's inequality $\|\nabla f\|^2 \geq \lambda_P\|f - \overline{f}\|^2$ on the domain $\Omega$.

\end{proof}

\begin{remark}[Higher regularity]
By using the $L^{\infty}$ bounds on $a$ and $b$, we can actually obtain exponential convergence to equilibrium in any $L^p$-norm with $1\leq p <+\infty$ via interpolation. Moreover, following e.g. \cite{DeFe08}, one can verify that any Sobolev norm of classical solutions to systems \eqref{2x2_0} (which matches the assumed regularity of the boundary $\partial\Omega$) will grow at most polynomially-in-time, so that exponential convergence  to equilibrium in a (slightly lower) Sobolev norm follows again via interpolation with \eqref{ExpoDecay2x2}.  
\end{remark}	

\begin{remark}[Generalised $2\times2$ systems]
The above arguments can be generalised to obtain exponential convergence to equilibrium for the following class of systems
	\begin{equation*}
	\begin{cases}
	a_t - d_a\Delta a = -\varphi(a)(r(a) - r(b)), &x\in\Omega, \quad t>0,\\
	b_t - d_b\Delta b = \varphi(a)(r(a) - r(b)), &x\in\Omega, \quad t>0,
	\end{cases}
	\end{equation*}
	subject to homogeneous Neumann boundary condition and initial data satisfying the bounds \eqref{bbo}, where $\varphi: [0,+\infty) \rightarrow [0,+\infty)$ is continuously increasing and $r: [0,+\infty) \rightarrow [0,+\infty)$ and its derivative $r'$ are continuously increasing.
%
	
	As example, we can consider reactions of the form 
	\begin{equation*}
	(n+m)\mathcal{A} \leftrightharpoons n\mathcal{A} + m\mathcal{B}
	\end{equation*}
	with $n\geq 0$ and $m\geq 1$, which are described by the mass action law systems 
	\begin{equation*}
	\begin{cases}
	a_t - d_a\Delta a = -m\,a^n(a^m - b^m), &x\in\Omega, \quad t>0,\\
	b_t - d_b\Delta b = m\,a^n(a^m - b^m), &x\in\Omega, \quad t>0,
	\end{cases}
	\end{equation*}
	corresponding to the cases $\varphi(z) = mz^n$ and $r(z) = z^m$.
\end{remark}
\bigskip

We now consider the
 complex-balanced chemical reaction network defined by \eqref{3x3_0}.
In strict contrast to the $2\times2$ system \eqref{2x2_0}, 
general reaction-diffusion systems like \eqref{3x3_0} do not allow to show uniform-in-time propagation of upper and positive lower {\it{a priori}} $L^{\infty}$-estimates.

The existence of global classical solutions to \eqref{3x3_0}
in space dimensions $N\leq 5$ follows from bootstrap arguments (in the spirit of e.g. \cite{Rot}) thanks to the specific structure of \eqref{3x3_0}, and since the nonlinear reaction terms are at most quadratic.
In higher space dimensions $N\ge6$, existence of global classical solutions can be shown under a ``closeness'' condition on the diffusion coefficients by following e.g. \cite{CDF14}.   
 The (rather standard) bootstrap argument along with the complete proof of the following Lemma stating existence of global, classical solutions is given in the Appendix. 
\begin{lemma}[Global classical solutions]\label{ClassicalSolution}\hfill\\
Let $\Omega$ be a bounded smooth ($C^2$) domain of $\mathbb{R}^N$, and  $d_a,d_b,d_c>0$, $k_1, k_2, k_3 >0$.
Assume nonnegative initial data $a_0, b_0, c_0\in L^{\infty}(\Omega)$,
and denote  
\begin{equation*}
\delta:= \max\{d_a,d_b,d_c\} - \min\{d_a,d_b,d_c\}.
\end{equation*}
Consider either $1\leq N\leq 5$ or $N\geq 6$ and $\delta>0$  sufficiently small (depending on $N$ as specified in the Appendix). 

Then, there exists a unique, nonnegative, global classical solution $(a,b,c)$ to system \eqref{3x3_0}, which satisfies the following $L^{\infty}$-bound for all $T>0$:
\begin{equation*}
\|a\|_{L^{\infty}(\Omega_T)} + \|b\|_{L^{\infty}(\Omega_T)} + \|c\|_{L^{\infty}(\Omega_T)} \leq C(T),
\end{equation*}
where $C(T)$ grows at most polynomially with respect to $T$.

Moreover, $a, b, c$ satisfy the conservation of mass 
\begin{equation}\label{mass3x3}
\int_{\Omega}(2\,a(x,t) + b(x,t) + c(x,t))\,dx = M:= \int_{\Omega}(2\,a_0(x) + b_0(x) + c_0(x))\, dx \quad \text{ for all } \quad t>0.
\end{equation}
\end{lemma}
		

Considering system \eqref{3x3_0}, we observe that any equilibrium $(a_{\infty}, b_{\infty}, c_{\infty})$ solves 
the stationary state conditions
\begin{equation*}
\begin{cases}
-k_1a_{\infty} + k_3b_{\infty}^2 = 0,\\
	k_1a_{\infty} + k_2b_{\infty}c_{\infty} - 2k_3b_{\infty}^2 = 0,\\
	k_1a_{\infty} - k_2b_{\infty}c_{\infty} = 0,
\end{cases}
\end{equation*}
and satisfies the mass conservation law
\begin{equation*}
	2a_{\infty} + b_{\infty} + c_{\infty} = M.
\end{equation*}
Thus, it follows that system \eqref{3x3_0} features two equilibria, first the unique positive complex-balance equilibrium
\begin{equation}\label{defeq} 
(a_{\infty}, b_{\infty}, c_{\infty}) \qquad \text{where}\quad a_{\infty} = \frac{k_3}{k_1}b_{\infty}^2, \ b_{\infty} = \frac{-k_1+\sqrt{k_1(k_1+2k_2M)}}{2k_2} \  \text{ and } \ c_{\infty} = \frac{k_3}{k_2}b_{\infty},
\end{equation}
and secondly the boundary equilibrium
\begin{equation*}
(a^*, b^*, c^*) = (0,0,M).
\end{equation*}
	
For notational convenience, we shall write
 $\mathbf{c} = (a,b,c)$ and $\mathbf{c}_{\infty} = (a_{\infty}, b_{\infty}, c_{\infty})$.	
At some places we shall also use (coherently with the first sections of this paper)  $c_1 := a$, $c_2 := b$ and $c_3 := c$ (and similarly for $c_{i,\infty}$ with $i=1,2,3$). 

The relative entropy writes as (cf. \eqref{FreeEnergy_PDE})
\begin{equation}\label{efin}
	\mathcal{E}(\mathbf{c}|\mathbf{c}_{\infty}) = \sum_{i=1}^{3}\int_{\Omega}\left(c_i\log{\frac{c_i}{c_{i,\infty}}} - c_i + c_{i,\infty}\right)dx,
\end{equation}
and its entropy dissipation is given by (cf. Prop. \ref{cale})
\begin{equation}\label{newd}
	\begin{aligned}
		\mathcal{D}(\mathbf{c}) 
			&= d_a\int_{\Omega}\frac{|\nabla a|^2}{a}dx + d_b\int_{\Omega}\frac{|\nabla b|^2}{b}dx + d_c\int_{\Omega}\frac{|\nabla c|^2}{c}dx\\
			&\quad + \int_{\Omega}\left[k_1a_{\infty}\Psi\left(\frac{a}{a_{\infty}}; \frac{bc}{b_{\infty}c_{\infty}}\right) + k_2b_{\infty}c_{\infty}\Psi\left(\frac{bc}{b_{\infty}c_{\infty}}; \frac{b^2}{b_{\infty}^2}\right)
			               + k_3b_{\infty}^2\Psi\left(\frac{b^2}{b_{\infty}^2}; \frac{a}{a_{\infty}}\right)\right]dx .
	\end{aligned}
\end{equation}
	
The following uniform-in-time {\it{a priori}} estimates for the average $\overline{c_i}(t)= \int_{\Omega} c_i(x,t)\,dx$ are a  natural consequence of the mass conservation law  \eqref{mass3x3} and the nonnegativity of the solutions:
\begin{equation}\label{mass}
2 \overline{a}(t)+ \overline{b}(t) + \overline{c}(t)  = M \quad {\hbox{for all}} \quad  t \ge 0.
\end{equation}

It is easy to see that if $\mathbf{c}_k = (a_k,b_k,c_k)$ is a space homogeneous sequence satisfying \eqref{mass} and converging to $(0,0,M)$, then
\begin{equation}\label{compute}
	\lim_{k\rightarrow \infty}\mathcal{D}(\mathbf c_k) = 0 \quad \text{ and } \quad \lim_{k\rightarrow \infty} \mathcal{E}(\mathbf c_k|\ww) = a_{\infty} + b_{\infty} + \Psi(M;c_{\infty}) >0,
\end{equation}
that means there does not exist any $\lambda >0$ such that $\mathcal{D}(\cc) \geq \lambda \mathcal{E}(\cc|\ww)$ for all $\cc$ satisfying the mass conservation \eqref{mass}. Hence, the convergence of solutions to \eqref{3x3_0} towards the strictly positive complex balance equilibrium is not clear. In fact, similarly to the Cor.\ref{2x2cor}, the following example shows that there is a non-trivial class of initial data, for which solutions to \eqref{3x3_0} will converge to the boundary equilibrium. 
\begin{remark}[Convergence towards boundary equilibrium]\label{r44}\hfill\\
Consider initial data $a_0 = b_0 \equiv 0$ and $\overline{c_0} = M$, yet with $c_0\not\equiv M$ (so that the initial data $(a_0,b_0,c_0)$ is not the boundary equilibrium $(0,0,M)$). 

By arguing similarly to the second part of Cor. \ref{2x2cor}, we obtain that 
 $a(t) = b(t) \equiv 0$ (for all $t \ge 0$), while $c:= c(t)$ is given by the solution of the homogeneous heat equation (with homogeneous Neumann boundary conditions), and thus converges exponentially fast to $M$ as
\begin{equation*}
\|c(t) - M\|^2 \leq e^{-d_c\lambda_P t}\|c_0 - M\|^2,
\end{equation*}
where $\lambda_P$ is the constant in Poincar\'{e}-Wirtinger's inequality $\|\nabla f\|^2 \geq \lambda_P\|f - \overline{f}\|^2$ on the domain $\Omega$.
	
\end{remark}

We now start the
\begin{proof}[Proof of Theorem \ref{ExpConverge}]
	We consider in this proof a classical solution $\cc := \cc(x,t)$ of system
	\eqref{3x3_0}.
	\par 
We first use the assumption $\|\frac{1}{b_0}\|_{L^{\infty}(\Omega)}<\infty$
to show that 
\begin{equation}\label{hh}
\inf_{x\in\Omega}b(x,t) \geq h(t):= \frac{1}{\bigl\|\frac{1}{b_0}\bigr\|_{L^{\infty}} + 2k_3t}, \qquad \text{for all}\quad t \ge 0.
\end{equation}
Indeed thanks to 
 the continuity of classical solutions, there exists a $t^*>0$ such that 
$\|b(t, \cdot)^{-1}\|_{L^{\infty}} \leq 2\,\|b_0^{-1}\|_{L^{\infty}}$ for all $t\in [0,t^*]$. 
Thus, for all $t\in [0,t^*]$, we can compute 
\begin{equation*}
\partial_t\left(\frac 1b\right) - d_b\,\Delta \left(\frac 1b\right) = -\frac{k_1a}{2b^2} - \frac{k_2c}{b} +2k_3 - 2d_b\frac{|\nabla b|^2}{b^3} \leq 2k_3.
\end{equation*}
Thus, using the maximum principle, we obtain
\begin{equation}\label{bi}
\left\|\frac{1}{b(t)}\right\|_{L^{\infty}(\Omega)} \leq \left\|\frac{1}{b_0}\right\|_{L^{\infty}(\Omega)} + 2k_3 t,\qquad \text{for all}\quad t\in [0,t^*],
\end{equation}
which implies the desired estimate for the time-interval $[0,t^*]$. Furthermore, it is easy to see that this argument can be iterated in time, and eventually yields the bound \eqref{bi}
for all 
$t>0$.
\medskip
	
We see therefore that if $b$ decays to zero, then it decays at most with a rate of the form $(1+t)^{-1}$. We will see that this lower bound,
 in combination with the entropy entropy-dissipation structure, is sufficient to show
  that it is impossible for the solution to converge to the boundary equilibrium.
	

As previously,  $\mathcal{E}$ defined by \eqref{efin} satisfies 
\begin{equation}\label{newadd}
\mathcal{E}(\cc|\ww) = \mathcal{E}(\cc|\overline{\cc}) 
+ \mathcal{E}(\overline{\cc}|\ww),
\end{equation}
 and via the Logarithmic Sobolev Inequality, we get 
\begin{equation}\label{add}
	\frac{1}{2}\mathcal{D}(\cc) \geq \lambda_1\mathcal{E}(\cc|\overline{\cc}).
\end{equation}
We denote  
$
	F(\cc) = \lambda_1\sum_{i=1}^{3}\Psi(c_i; c_{i,\infty}).
$
By using estimate \eqref{bi} (for all $t \ge 0$), the Logarithmic Sobolev Inequality
 and the additivity property \eqref{newadd}, we see that (with $h(t)$ being defined in
 \eqref{hh}) 
		\begin{equation*}
			\begin{aligned}
				\frac{1}2 \mathcal{D}(\mathbf{c}) &\geq \lambda_1(\mathcal{E}(\cc|\ww) - \mathcal{E}(\overline{\cc}|\ww))\\
				&\qquad  + \frac 12 \int_{\Omega}\left[k_1a_{\infty}\Psi\left(\frac{a}{a_{\infty}}; \frac{bc}{b_{\infty}c_{\infty}}\right) + k_2c_{\infty}h(t)\Psi\left(\frac{c}{c_{\infty}}; \frac{b}{b_{\infty}}\right)
							               + k_3b_{\infty}^2\Psi\left(\frac{b^2}{b_{\infty}^2}; \frac{a}{a_{\infty}}\right)\right]dx\\
				&\geq \alpha h(t) \bigg\{ \int_{\Omega}\underbrace{\left[F(\cc) + \Psi\left(\frac{a}{a_{\infty}}; \frac{bc}{b_{\infty}c_{\infty}}\right) + \Psi\left(\frac{c}{c_{\infty}}; \frac{b}{b_{\infty}}\right) 
								+ \Psi\left(\frac{b^2}{b_{\infty}^2}; \frac{a}{a_{\infty}}\right)\right]}_{:=\Phi(\cc)}dx - F(\overline{\cc}) \bigg\} \\
				&\geq \alpha \,h(t)\, \left(\int_{\Omega}\widehat{\Phi}(\cc)\, dx - F(\overline{\cc})\right)\\
				&\geq \alpha\, h(t)\, (\widehat{\Phi}(\overline{\cc}) - F(\overline{\cc})),
\end{aligned}		
\end{equation*}
where 
$\alpha = \left\|\frac{1}{b_0}\right\|_{L^{\infty}}\min\left\{ \lambda_1;\; \frac 12 k_1a_{\infty};\; \frac 12 k_2c_{\infty}\left\|\frac{1}{b_0}\right\|_{L^{\infty}}^{-1}; \; \frac 12 k_3b_{\infty}^2 \right\}$,
$\Phi(\cc) = F({\cc}) +Q(\cc) $ and 
$$Q(\cc):= \Psi\left(\frac{a}{a_{\infty}}; \frac{bc}{b_{\infty}c_{\infty}}\right) + \Psi\left(\frac{c}{c_{\infty}}; \frac{b}{b_{\infty}}\right)+ \Psi\left(\frac{b^2}{b_{\infty}^2}; \frac{a}{a_{\infty}}\right) .
$$
We also recall that $\widehat{\Phi}$ denotes the convexification of $\Phi$. Define 
$$\mathfrak{C}_M = \{\boldsymbol{\xi} = (a,b,c)\in \mathbb{R}^3_{\geq 0}: 2a + b + c = M \quad \text{ and } \quad \mathcal{E}(\boldsymbol{\xi}|\ww) \leq \mathcal{E}(\cc_0|\ww)\}.$$
the set of all concentrations satisfying the mass conservation \eqref{mass3x3} and suitably bounded relative entropy. Note that 
$$
	\boldsymbol{\xi} \in \mathfrak{C}_M \quad \text{ and } \quad Q(\boldsymbol{\xi})=   0 \qquad \Longleftrightarrow \qquad \boldsymbol{\xi}\equiv \ww ,
$$
and
$$
	\widehat{\Phi}(\boldsymbol{\xi}) - F(\boldsymbol{\xi}) \geq \widehat{F}(\boldsymbol{\xi}) + \widehat{Q}(\boldsymbol{\xi}) - F(\boldsymbol{\xi}) \geq \widehat{Q}(\boldsymbol{\xi}) \geq 0.
$$
Hence, 
\begin{equation*}
	\inf_{\boldsymbol{\xi}\in \mathfrak{C}_M}\frac{\widehat{\Phi}(\boldsymbol{\xi}) - F(\boldsymbol{\xi})}{\mathcal{E}(\boldsymbol{\xi}|\ww)} > 0 \quad \text{ if } \quad \liminf_{\mathfrak{C}_M\ni \boldsymbol{\xi} \rightarrow \ww}\frac{\widehat{\Phi}(\boldsymbol{\xi}) - F(\boldsymbol{\xi})}{\mathcal{E}(\boldsymbol{\xi}|\ww)} > 0.
\end{equation*}
On the other hand, we can estimate by using Lemma \ref{lem:conv} and Taylor expansion around $\ww$ (see Proposition \ref{ODE_Theorem} or \cite[Proposition 3.3]{MiHaMa14})
\begin{equation*}
	\liminf_{\mathfrak{C}_M\ni \boldsymbol{\xi} \rightarrow \ww}\frac{\widehat{\Phi}(\boldsymbol{\xi}) - F(\boldsymbol{\xi})}{\mathcal{E}(\boldsymbol{\xi}|\ww)} \ge \liminf_{\mathfrak{C}_M\ni \boldsymbol{\xi} \rightarrow \ww}\frac{{\widehat{Q}}(\boldsymbol{\xi})}{\mathcal{E}(\boldsymbol{\xi}|\ww)} = 
	\liminf_{\mathfrak{C}_M\ni \boldsymbol{\xi} \rightarrow \ww}\frac{Q(\boldsymbol{\xi})}{\mathcal{E}(\boldsymbol{\xi}|\ww)} > 0.
\end{equation*}
Finally, we see that
\begin{equation*}
	\inf_{\boldsymbol{\xi}\in \mathfrak{C}_M}\frac{\widehat{\Phi}(\boldsymbol{\xi}) - F(\boldsymbol{\xi})}{\mathcal{E}(\boldsymbol{\xi}|\ww)} \geq \beta >0,
\end{equation*}
and consequently
\begin{equation}\label{nonExplicit}
			\frac 12\mathcal{D}(\cc) \geq \alpha\beta h(t)\mathcal{E}(\overline{\cc}|\ww),
		\end{equation}
		which, combining with \eqref{add}, implies 
		\begin{equation}\label{gammaExplicit}
			\mathcal{D}(\cc) \geq \gamma h(t)\mathcal{E}(\cc|\ww) \quad \text{ for all } \quad t>0,
		\end{equation}
		where $\gamma = \min\left\{\alpha\beta; \lambda_1\left\|\frac{1}{b_0}\right\|_{L^{\infty}(\Omega)} \right\}$.
Then, thanks to Gronwall's lemma, we see that for all $t \ge 0$,
		\begin{equation*}
			\mathcal{E}(\cc(t)|\ww) \leq \mathcal{E}(\cc_0|\ww){\left(\left\|\frac{1}{b_0}\right\|_{L^{\infty}(\Omega)} + 2k_3t\right)^{-\gamma/2k_3}} 
		\end{equation*}
and the relative entropy with respect to the complex balance equilibrium decays therefore
 to zero with the algebraic rate $\gamma/2k_3$.
	Moreover, by using 
a Csisz\'{a}r-Kullback-Pinsker inequality (see e.g. Lemma \ref{CKPinequality}), we get the
estimate
\begin{equation*}
		\|a(t) - a_{\infty}\|_{L^1(\Omega)}^2 + \|b(t) - b_{\infty}\|_{L^1(\Omega)}^2 + \|c(t) - c_{\infty}\|_{L^1(\Omega)}^2 \leq \frac{\mathcal{E}(\cc_0|\ww)}{C_{CKP}}{\left(\left\|\frac{1}{b_0}\right\|_{L^{\infty}(\Omega)} + 2k_3t\right)^{-\gamma/2k_3}}.
	\end{equation*}
As a consequence, there exists for any (sufficiently small) $ \varepsilon>0$ a time $T_{\varepsilon}>0$ such that solutions are bounded away from the boundary equilibrium in the sense of $L^1$, i.e. $\|a(t)\|_{L^1(\Omega)} \geq \varepsilon^2$, $\|b(t)\|_{L^1(\Omega)} \geq \varepsilon^2$ and $\|c(t)\|_{L^1(\Omega)} \geq \varepsilon^2$ for all $t\geq T_{\varepsilon}$. These $L^1$-bounds away from the boundary equilibrium allow to apply a specialised entropy entropy-dissipation estimate proven in Lemma \ref{3x3_eede} below. Finally, via another Gronwall arguments, we obtain
	\begin{equation*}
		\mathcal{E}(\cc(t)|\ww) \leq e^{-\lambda_{\varepsilon} t}\mathcal{E}(\cc_0|\ww) \qquad \text{ for all }\quad t\geq T_{\varepsilon}
	\end{equation*}
	which consequently, together with another use of a Csisz\'{a}r-Kullback-Pinsker inequality, implies 
	Theorem \ref{ExpConverge}.		
\end{proof}	

\begin{remark}\label{rmk:explicit}
We remark that since we prove estimate \eqref{nonExplicit} via the convexification technique, we get 
a non-explicit constant $\gamma$ in \eqref{gammaExplicit}. In fact, $\gamma$  can be explicitly estimated by using the constructive method in Subsection \ref{sec:constructive}, at the price of a much longer proof.
\end{remark}

	\begin{lemma}[Entropy entropy-dissipation estimate]\label{3x3_eede}
Consider a positive initial mass $M>0$. 

Then, for any nonnegative, measurable functions $a, b, c:\Omega \rightarrow \mathbb{R}_+$ satisfying the mass conservation law
\begin{equation*}
\int_{\Omega}(2a(x) + b(x) + c(x))\, dx = M,
\end{equation*}
and the lower bounds $\|a\|_{L^1(\Omega)} \geq \varepsilon^2$, $\|b\|_{L^1(\Omega)} \geq \varepsilon^2$ and $\|c\|_{L^1(\Omega)} \geq \varepsilon^2$,
 the functional inequality
\begin{equation}\label{e_estimate}
\mathcal D(\mathbf{c}) \geq \lambda_{\varepsilon}\,\mathcal{E}(\mathbf{c}|\mathbf{c}_{\infty})
\end{equation}
holds, for some explicit constant 
 $\lambda_{\varepsilon}>0$, 
  which depends only on the initial mass $M$, the domain $\Omega$, the diffusion coefficients $d_a, d_b, d_c$, the reaction rates $k_1$, $k_2$, $k_3$, and on $\varepsilon$. Here $\mathcal{E}(\mathbf c|\mathbf c_{\infty})$ and $\mathcal D(\mathbf c)$ are defined by \eqref{efin} and \eqref{newd}, while
  $\ww$ is defined by \eqref{defeq}. Moreover, we have
		\begin{equation*}
		\lambda_{\varepsilon} = O(\varepsilon^6) \quad \text{ as } \quad \varepsilon\rightarrow 0.
		\end{equation*}
	\end{lemma}
	
\begin{proof} 
We will apply the  method described in Subsection \ref{sec:constructive}.
We already know that \eqref{add} (that is, Step 1), holds.
\medskip 

In Step 2, we bound the entropy dissipation $\mathcal{D}(\cc)$ below and the relative entropy $\mathcal{E}(\cc|\ww)$ above by quadratic terms for the square root concentrations $\CC := (A,B,C)$:
\begin{equation*}
		\mathcal{D}(\cc)\geq  K_0\biggl(\|\nabla A\|^2 + \|\nabla B\|^2 + \|\nabla C\|^2
		 + \biggl\|\frac{{A}}{A_{\infty}} - \frac{{B}^2}{B_{\infty}^2}\biggr\|^2 + \biggl\|\frac{{B}^2}{B_{\infty}^2} - \frac{{B}{C}}{B_{\infty}C_{\infty}}\biggr\|^2 + \biggl\|\frac{{B}{C}}{B_{\infty}C_{\infty}} - \frac{{A}}{A_{\infty}}\biggr\|^2\biggr), 
\end{equation*}
		and
		\begin{equation*}
		\mathcal{E}(\overline{\cc}|\ww) \leq K_2\left(\left(\sqrt{\overline{A^2}} - A_{\infty}\right)^2 + \left(\sqrt{\overline{B^2}} - B_{\infty}\right)^2 +\left(\sqrt{\overline{C^2}} - C_{\infty}\right)^2\right).
		\end{equation*}

%
%
Step 3 enables to estimate the entropy dissipation $\mathcal{D}(\cc)$
 in terms of reaction terms for averaged quantities:
\begin{equation*}
\mathcal{D}(\cc)\geq  K_1\biggl[\|\nabla A\|^2 + \|\nabla B\|^2 + \|\nabla C\|^2 + \biggl(\frac{\overline{A}}{A_{\infty}} - \frac{\overline{B}^2}{B_{\infty}^2}\biggr)^2 + \biggl(\frac{\overline{B}^2}{B_{\infty}^2} - \frac{\overline{B}\,\overline{C}}{B_{\infty}C_{\infty}}\biggr)^2 + \biggl(\frac{\overline{B}\,\overline{C}}{B_{\infty}C_{\infty}} - \frac{\overline{A}}{A_{\infty}}\biggr)^2\biggr] . 
\end{equation*}
It remains to find $K_3>0$ such that
\begin{multline}\label{g5}
\|\nabla A\|^2 + \|\nabla B\|^2 + \|\nabla C\|^2 + \biggl(\frac{\overline{A}}{A_{\infty}} - \frac{\overline{B}^2}{B_{\infty}^2}\biggr)^2 + \biggl(\frac{\overline{B}^2}{B_{\infty}^2} - \frac{\overline{B}\,\overline{C}}{B_{\infty}C_{\infty}}\biggr)^2 + \biggl(\frac{\overline{B}\,\overline{C}}{B_{\infty}C_{\infty}} - \frac{\overline{A}}{A_{\infty}}\biggr)^2\\
\geq K_3\left(\left(\sqrt{\overline{A^2}} - A_{\infty}\right)^2 + \left(\sqrt{\overline{B^2}} - B_{\infty}\right)^2 +\left(\sqrt{\overline{C^2}} - C_{\infty}\right)^2\right).
\end{multline}
We exploit the ansatz 
\begin{equation*}
\overline{A^2} = A_{\infty}^2(1+\mu_A)^2, \quad \overline{B^2} = B_{\infty}^2(1+\mu_B)^2, \quad \overline{C^2} = C_{\infty}^2(1+\mu_C)^2, \qquad \text{with} \quad
\mu_A, \mu_B, \mu_C\in [-1,+\infty). 
\end{equation*}
Thanks to the natural upper bounds for $\overline{A^2}, \overline{B^2}$ and $\overline{C^2}$, we have more precisely
		\begin{equation*}
		-1\leq \mu_A, \mu_B, \mu_C \leq \mu_{\max} <+\infty ,
		\end{equation*}
		for some $\mu_{\max}>0$. 
		We again denote  $\delta_{A}(x) = A(x) - \overline{A}$, $\delta_B(x) = B(x) - \overline{B}$ and $\delta_C(x) = C(x) - \overline{C}$, and recall that
		\begin{equation*}
		\overline{A} = A_{\infty}(1+\mu_A) - \frac{\|\delta_A\|^2}{\sqrt{\overline{A^2}} + \overline{A}} =: A_{\infty}(1+\mu_A) - R(A)\|\delta_A\|^2 .
		\end{equation*}
Similarly,
		\begin{equation*}
		\overline{B} = B_{\infty}(1+\mu_B) - R(B)\|\delta_B\|^2 \quad \text{ and } \quad \overline{C} = C_{\infty}(1+\mu_C) - R(C)\|\delta_C\|^2.
		\end{equation*}
Thanks to the assumption $\overline{A^2} \geq \varepsilon^2$, $\overline{B^2} \geq \varepsilon^2$ and $\overline{C^2} \geq \varepsilon^2$, we see that
$$
R(A), \;R(B), \;R(C) \leq \frac{1}{\varepsilon}.
$$ 
By using computations similar to \eqref{t1}, we get 
	\begin{equation*}
		\begin{aligned}
			&\left(\frac{\overline{A}}{A_{\infty}} - \frac{\overline{B}^2}{B_{\infty}^2}\right)^2
			= \left[(1+\mu_A) - \frac{R(A)\|\delta_A\|^2}{A_{\infty}} - \left((1+\mu_B) - \frac{R(B)\|\delta_B\|^2}{\textcolor{blue}{B_{\infty}}}\right)^2\right]^2\\
			&\quad \geq \frac 12 \left[(1+\mu_A) - (1+\mu_B)^2\right]^2
			- 3\frac{\|\delta_A\|^4}{A_{\infty}^2}R(A)^2 - 3\frac{4\|\delta_B\|^4(1+\mu_B)^2}{\textcolor{black}{B_{\infty}^2}}R(B)^2 - 3\frac{\|\delta_B\|^8}{\textcolor{black}{B_{\infty}^4}}R(B)^4\\
			&\quad \geq \frac 12 \left[(1+\mu_A) - (1+\mu_B)^2\right]^2 - 3\frac{\|\delta_A\|^2}{A_{\infty}^2} \frac{\|\delta_A\|^2}{\varepsilon^2} - 3\frac{4\|\delta_B\|^2(1+\mu_B)^2}{\textcolor{black}{B_{\infty}^2}} \frac{\|\delta_B\|^2}{\varepsilon^2} - 3\frac{\|\delta_B\|^6}{\textcolor{black}{B_{\infty}^4}} \frac{\|\delta_B\|^2}{\varepsilon^4}			.
		\end{aligned}
	\end{equation*}
	Then, thanks to the boundedness of $\|\delta_A\|$ and $\|\delta_B\|$ and $1/\varepsilon^2 \leq 1/\varepsilon^4$ for  $\varepsilon \in ]0,1]$, we can estimate
	\begin{equation*}
		\begin{aligned}
			\left(\frac{\overline{A}}{A_{\infty}} - \frac{\overline{B}^2}{B_{\infty}^2}\right)^2
			\geq \frac 12 \left[(1+\mu_A) - (1+\mu_B)^2\right]^2 - \frac{K}{\varepsilon^4}\left(\|\delta_A\|^2 + \|\delta_B\|^2\right),
		\end{aligned}
	\end{equation*}
	where $K$ is independent of $\varepsilon$. Similarly,
	\begin{equation*}
		\left(\frac{\overline{B}^2}{B_{\infty}^2} - \frac{\overline{B}\,\overline{C}}{B_{\infty}C_{\infty}}\right)^2 \geq \frac{1}{2}(1+\mu_B)^2(\mu_B - \mu_C)^2 - \frac{K}{\varepsilon^4}(\|\delta_B\|^2 + \|\delta_C\|^2)
	\end{equation*}
	and
	\begin{equation*}
		\left(\frac{\overline{B}\,\overline{C}}{B_{\infty}C_{\infty}} - \frac{\overline{A}}{A_{\infty}}\right)^2 \geq \frac 12\left[(1+\mu_B)(1+\mu_C) - (1+\mu_A)\right]^2 - \frac{K}{\varepsilon^4}(\|\delta_A\|^2 + \|\delta_B\|^2 + \|\delta_C\|^2).
	\end{equation*}
	The constant $K$ above may vary  but is always {\it independent of $\varepsilon$.}
	Therefore, the left hand side of \eqref{g5} is bounded below by
			\begin{equation*}
			\begin{aligned}
			K_4(\varepsilon)\biggl(\left((1+\mu_A) - (1+\mu_B)^2\right)^2 + (1+\mu_B)^2\left(\mu_B - \mu_C\right)^2+ \left((1+\mu_B)(1+\mu_C) - (1+\mu_A)\right)^2\biggr),
			\end{aligned}
			\end{equation*}
			with $K_4(\varepsilon) = O(\varepsilon^4)$, 
			while the right hand side of \eqref{g5} is bounded above by
		\begin{equation*}
		K_3\max\{\Ainf^2, \Binf^2, \Cinf^2\}\left(\mu_A^2 + \mu_B^2 + \mu_C^2\right).
		\end{equation*}
Therefore, we only have to show for some $K_5>0$ (and $K_5 = O(\varepsilon^2)$ as $\varepsilon\rightarrow 0$), that
\begin{multline}\label{g6}
\left((1+\mu_A) - (1+\mu_B)^2\right)^2 + (1+\mu_B)^2\left(\mu_B - \mu_C\right)^2+ \left((1+\mu_B)(1+\mu_C) - (1+\mu_A)\right)^2\\
		\geq K_5\left(\mu_A^2 + \mu_B^2 +\mu_C^2\right),
\end{multline}
under the constraint imposed by the mass conservation law
\begin{equation}\label{massmu}
2 A_{\infty}^2\,\mu_A(\mu_A + 2) + B_{\infty}^2\,\mu_B(\mu_B + 2) + C_{\infty}^2\,\mu_C(\mu_C + 2) = 0.
\end{equation}
To prove \eqref{g6}, we first use $\overline{B^2} \geq \varepsilon^2$, which leads to {$(1 + \mu_B)^2 \geq \frac{\varepsilon^2}{\Binf^2}$} and
\begin{equation}\label{1star}
	(1+\mu_B)^2(\mu_B - \mu_C)^2 \geq \frac{\varepsilon^2}{\Binf^2}(\mu_B - \mu_C)^2.
\end{equation}
The mass conservation law \eqref{massmu} implies then that only 
the following three cases concerning the signs of $\mu_A, \mu_B$ and $\mu_C$ can appear:
\begin{itemize}
\item[(I)] $\mu_A$ and $\mu_B$ have different signs,
\item[(II)] $\mu_A \geq 0$, $\mu_B \geq 0$ and $\mu_C\leq 0$,
\item[(III)] $\mu_A \leq 0$, $\mu_B \leq 0$ and $\mu_C \geq 0$.
\end{itemize}
We will treat each case separately.
		
\medskip		
\noindent \underline{Case (I): {\it $\mu_A$ and $\mu_B$ have different signs}}. In this case, we see that
			\begin{equation}\label{2star}
				\left[(1+\mu_A) - (1+\mu_B)^2\right]^2 \geq \mu_A^2 + \mu_B^2.
			\end{equation}
			Indeed, if $\mu_A\geq 0$ and $\mu_B \leq 0$,
			\begin{equation*}
				\left[(1+\mu_A) - (1+\mu_B)^2\right]^2 = (\mu_A - \mu_B - \mu_B(1+\mu_B))^2 \geq (\mu_A - \mu_B)^2 \geq \mu_A^2 + \mu_B^2
			\end{equation*}
			since $\mu_A - \mu_B \geq 0$ and $-\mu_B(1+\mu_B)\geq 0$.			
			If $\mu_A\leq 0$ and $\mu_B \geq 0$, then 
			\begin{equation*}
				\left[(1+\mu_A) - (1+\mu_B)^2\right]^2 = (\mu_B - \mu_A + \mu_B(1+\mu_B))^2 \geq (\mu_B - \mu_A)^2 \geq \mu_B^2 + \mu_A^2
			\end{equation*}
			From \eqref{1star} and \eqref{2star}, we get
			\begin{equation*}
				\text{LHS of } \eqref{g6} \geq \mu_A^2 + \mu_B^2 + \frac{\varepsilon^2}{B_{\infty}^2}(\mu_B - \mu_C)^2 \geq \frac{1}{4}\min\left\{\frac{\varepsilon^2}{B_{\infty}^2}; 1\right\}(\mu_A^2 + \mu_B^2 + \mu_C^2),
			\end{equation*}
			 and thus obtain \eqref{g6} with	
$$
K_5 = \frac{1}{4}\min\left\{\frac{\varepsilon^2}{\Binf^2}; 1\right\}.
$$
		
		\medskip
		\noindent \underline{Case (II): {\it $\mu_A \geq 0$, $\mu_B \geq 0$ and $\mu_C\leq 0$}}. In this case, we first estimate further \eqref{1star} as
		\begin{equation}\label{1star_bis}
			(1+\mu_B)^2(\mu_B - \mu_C)^2 \geq {\frac{\varepsilon^2}{\Binf^2}}(\mu_B - \mu_C)^2 \geq \frac{\varepsilon^2}{\Binf^2}(\mu_B^2 + \mu_C^2).
		\end{equation}
		There are two possibilities: $\mu_A\geq \mu_B$ or $\mu_B \geq \mu_A$.	
			If $\mu_A \geq \mu_B$, then
			\begin{equation}\label{3star}
				\left[(1+\mu_B)(1+\mu_C) - (1+\mu_A)\right]^2 = (\mu_A - \mu_B - \mu_C(1+\mu_B))^2 \geq (\mu_A - \mu_B)^2.
			\end{equation}			
			If $\mu_B \geq \mu_A$, then 
			\begin{equation}\label{4star}
				\left[(1+\mu_A) - (1+\mu_B)^2\right]^2 = (\mu_B - \mu_A + \mu_B(1+\mu_B))^2 \geq (\mu_B - \mu_A)^2.
			\end{equation}
			Therefore, thanks to \eqref{1star_bis}, \eqref{3star} and \eqref{4star}, we can obtain \eqref{g6} in case (II) with
			$$
				K_5 = \frac{1}{4}\min\left\{\frac{\varepsilon^2}{\Binf^2}; 1\right\}.
			$$
		
		\medskip		
		\noindent \underline{Case (III): {\it $\mu_A \leq 0$, $\mu_B \leq 0$ and $\mu_C\geq 0$}}. This case can be treated in the same way as case (II) by using
		\begin{equation*}
			(1+\mu_B)^2(\mu_B - \mu_C)^2 \geq \frac{\varepsilon^2}{\Binf^2}(\mu_C - \mu_B)^2 \geq \frac{\varepsilon^2}{\Binf^2}(\mu_C^2 + \mu_B^2),
		\end{equation*}
		and considering $\mu_A \geq \mu_B$ or $\mu_A \leq \mu_B$. We get from this case \eqref{g6}, with again
			$$
				K_5 = \frac{1}{4}\min\left\{\frac{\varepsilon^2}{\Binf^2}; 1\right\}.
			$$
		
		\medskip
		In conclusion, we have proven \eqref{g6}, with
		$$
			K_5 = \frac{1}{4}\min\left\{\frac{\varepsilon^2}{\Binf^2}; 1\right\},
		$$
		and consequently finished the proof of Lemma \ref{3x3_eede}.
	\end{proof}
\begin{remark}
Obtaining a general stability statement of the strictly positive complex balance equilibrium for general (complex balanced) reaction-diffusion systems featuring boundary equilibria looks quite involved, and remains open for future investigation. We would just like to point out 
that the difficulty appears only when a trajectory $t \mapsto \ww(t)$
gets too close to a boundary equilibrium.

More precisely, let us consider a complex balanced system,
and denote 
\begin{equation*}
	\mathcal{B}_E = \{\cc^* \in \partial\mathbb{R}^N_{\geq 0}: \cc^* \text{ is a complex balance equilibrium of } \eqref{e0} \}
\end{equation*}
and
\begin{equation*}
	\mathcal{N}_{\varepsilon}(\mathcal B_E) = \{\boldsymbol{\xi}\in \mathbb{R}^N_{>0}: d(\boldsymbol{\xi}, \mathcal B_E) \geq \varepsilon\}
\end{equation*}
 an $\varepsilon$-neighbourhood of $\mathcal B_E$, where $d$ is the Euclidean distance in $\mathbb R^n$.
Let us assume that there exists $\varepsilon>0$  such that each trajectory of \eqref{e0}  starting outside of $\mathcal{N}_\varepsilon(\mathcal B_E)$ remains outside $\mathcal{N}_{\varepsilon}(\mathcal B_E)$ for all $t \ge 0$, 
 that is
	\begin{equation*}
		\overline{\cc}(t) \in \mathbb{R}^N_{>0}\backslash \mathcal{N}_{\varepsilon}(\mathcal B_E) \quad \text{ for all } \quad t \ge 0.
	\end{equation*}
	Then one can show that there exists a constant $\lambda_{\varepsilon}>0$ depending on the domain $\Omega$, the diffusion coefficients $\mathbb D$, the stoichiometric coefficients, and additionally  the parameter $\varepsilon$, such that
	\begin{equation*}
		\mathcal{D}(\cc(t)) \geq \lambda_{\varepsilon} \,\mathcal{E}(\cc(t)|\ww).
	\end{equation*}
\end{remark}

\section{Appendix}\label{appendix}
	In this Appendix, we prove Lemma \ref{ClassicalSolution} about the global existence of classical solution to the system \eqref{3x3_0} and Lemma \ref{lem:conv}. We first need the following classical lemma on the regularity of solutions to the heat equation.
	\begin{lemma}\label{HeatRegularity}[see e.g. \cite{CDF14}]
		Assume that $u_0\in L^{\infty}(\Omega)$, $f\in L^p(\Omega_T)$, $d>0$, and
		that  $u$ is the solution to
		\begin{equation*}
			\begin{cases}
				u_t - d\Delta u = f, &x\in\Omega, \quad t>0,\\
				\nabla u\cdot \nu = 0, &x\in\partial\Omega, \quad t>0,\\
				u(x,0) = u_0(x), &x\in\Omega.
			\end{cases}
		\end{equation*}
		If $1<p < \frac{N+2}{2}$, then $u\in L^{s}(\Omega_T)$ for all $1\leq s< \frac{p(N+2)}{N+2-2p}$. If $p\geq \frac{N+2}{2}$, then $u\in L^{\infty}(\Omega_T)$.
		Moreover, the corresponding norms grow at most polynomially w.r.t $T$.
	\end{lemma}
	\begin{proof}[Proof of Lemma \ref{ClassicalSolution}]
		When $1\leq N \leq 2$ or $N\geq 6$ and $\delta$ is small enough, the results follow from \cite{CDF14}.
		Here we prove that when $3 \leq N \leq 5$, the system \eqref{3x3_0} possesses global classical solution without assumption on  the diffusion coefficients.
		First we observe that
		\begin{equation*}
			(2a+b+c)_t - \Delta(2d_aa + d_bb + d_cc) = 0,
		\end{equation*}
		or equivalently
		\begin{equation*}
			z_t - \Delta(Mz) = 0,
		\end{equation*}
		with $z = 2a+b+c$, and 
		\begin{equation*}
			0< \min\{d_a,d_b,d_c\} \leq M(x,t):= \frac{2d_aa + d_bb + d_cc}{2a + b + c} \leq \max\{d_a,d_b,d_c\} < +\infty.
		\end{equation*}
		It follows from e.g. \cite{Pie10} that $z \in L^2(\Omega_T)$ for all $T>0$
		(with norms at most polynomially growing w.r.t $T$), which in combination with the nonnegativity of the concentrations leads to $a, b, c\in L^2(\Omega_T)$. Secondly, 
we obtain from \eqref{3x3_0} that
		\begin{subequations} 
			\begin{align}
				a_t - d_a\Delta a &\leq k_3 b^2,\label{ff1}\\
				b_t - d_b\Delta b &\leq k_1 a + k_2 bc,\label{ff2}\\
				c_t - d_c\Delta c &\leq k_1 a.\label{ff3}
			\end{align}
		\end{subequations}

The strategy of the proof is the following: By using Lemma \ref{HeatRegularity} and the regularity of $a$, we get from \eqref{ff3} an improved regularity of $c$. This yields an improved regularity of $a + bc$ and thus from \eqref{ff2} an improvement of the regularity of $b$. This information is used in \eqref{ff1} to deduce the improved regularity of $a$. Graphically, we have the following bootstrap iteration:
		
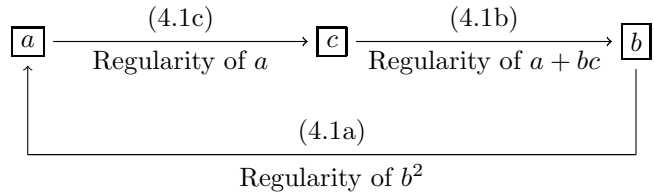
\begin{equation*}
			\begin{tikzpicture}
				\node (a) {$\boxed{a}$} node (b) at (4, 0) {$\boxed{c}$} node (c) at (8, 0) {$\boxed{b}$};
				\draw[arrows=->]  (a.east) -- node[above]{\eqref{ff3}} (b.west);
				\draw[arrows=->]  (a.east) -- node[below]{Regularity of $a$} (b.west);
				\draw[arrows=->]  (b.east) -- node[above]{\eqref{ff2}}  (c.west);
				\draw[arrows=->]  (b.east) -- node[below]{Regularity of $a+bc$}  (c.west);
				\draw (c.south) -- (8, -1.5);
				\draw (8, -1.5) -- node[above]{\eqref{ff1}} (0, -1.5);
				\draw (8, -1.5) -- node[below]{Regularity of $b^2$} (0, -1.5);
				\draw[arrows=->] (0, -1.5) -- (a.south);
			\end{tikzpicture}
		\end{equation*}
		This iteration terminates whenever we achieve $a, b, c \in L^{\infty}(\Omega_T)$. In the rest of the proof, we will show how this iteration works when $N = 3, 4, 5$ and also explain why it fails when $N\geq 6$. To avoid unnecessary long explanations, we will present the iterations for each case in a table. We shall use the notation $f\in L^{\alpha - 0}$ if $f\in L^p(\Omega_T)$ for all $p<\alpha$ and $f\in L^{\infty-0}$ if $f\in L^p(\Omega_T)$ for all $p<+\infty$ (with norms at most polynomially growing w.r.t $T$).
		
		\medskip
		\noindent{\underline{Case $N = 3$}}:
		\begin{center}
		  \begin{tabular}{ | l | c | c | c | c |}
		    \hline
		    				 & $a$ 			& $c$ 			& $a + bc$ & $b$\\ \hline
		    Step 0      & $L^2$     &  $L^2$      &      --        &  $L^2$\\ \hline				 
		    Step 1 		& $L^2$ 	&  $L^{10-0}$		& $L^{5/3 - 0}$ & $L^{5 - 0}$ \\ \hline
		    Step 2 		& $L^{\infty - 0}$ 	&  $L^{\infty}$		& $L^{5 - 0}$ & $L^{\infty}$ \\ \hline
		    Step 3 		& $L^{\infty}$ 	&  --		& -- & -- \\ \hline
		  \end{tabular}
		\end{center}
		
		\medskip
		\noindent{\underline{Case $N = 4$}}:
		\begin{center}
			  \begin{tabular}{ | l | c | c | c | c |}
			    \hline
			    				 & $a$ 			& $c$ 			& $a + bc$ & $b$\\ \hline
			    Step 0      & $L^2$     &  $L^2$      &      --        &  $L^2$\\ \hline
			    Step 1 		& $L^2$ 	&  $L^{6-0}$		& $L^{3/2 - 0}$ & $L^{3 - 0}$ \\ \hline
			    Step 2 		& $L^{3-0}$ 	&  $L^{\infty-0}$		& $L^{3 - 0}$ & $L^{\infty-0}$ \\ \hline
			    Step 3 		& $L^{\infty}$ 	&  $L^{\infty}$		& $L^{\infty-0}$ & $L^{\infty}$ \\ \hline
			  \end{tabular}
		\end{center}			
		
		\medskip
		\noindent{\underline{Case $N = 5$}}:
		\begin{center}
			  \begin{tabular}{ | l | c | c | c | c |}
			    \hline
			    				 & $a$ 			& $c$ 			& $a + bc$ & $b$\\ \hline
			    Step 0      & $L^2$     &  $L^2$      &      --        &  $L^2$\\ \hline
			    Step 1 		& $L^{2}$ 	&  $L^{14/3-0}$		& $L^{7/5 - 0}$ & $L^{7/3 - 0}$ \\ \hline
			    Step 2 		& $L^{2}$ 	&  $L^{14/3-0}$		& $L^{14/9 - 0}$ & $L^{14/5-0}$ \\ \hline
			    Step 3 		& $L^{\infty-0}$ 	&  $L^{\infty}$		& $L^{14/5-0}$ & $L^{\infty}$ \\ \hline
			    Step 4 		& $L^{\infty}$ 	&  --		& -- & -- \\ \hline
			  \end{tabular}
		\end{center}		
		
		\medskip
				\noindent{\underline{Case $N = 6$}}:
				\begin{center}
					  \begin{tabular}{ | l | c | c | c | c |}
					    \hline
					    				 & $a$ 			& $c$ 			& $a + bc$ & $b$\\ \hline
					    Step 0      & $L^2$     &  $L^2$      &      --        &  $L^2$\\ \hline
					    Step 1 		& $L^{2}$ 	&  $L^{4-0}$		& $L^{4/3 - 0}$ & $L^{2}$ \\ \hline
					    Step 2 		& $L^{2}$ 	&  $L^{4-0}$		& $L^{4/3 - 0}$ & $L^{2}$ \\
					    \hline
					  \end{tabular}
				\end{center}		
		In case $N=6$ (or similarly $N\geq 6$) after Step 1, one cannot further improve the regularity of $a$, $b$ or $c$, to continue the bootstrap argument.
	\end{proof}
	
	\begin{proof}[Proof of Lemma \ref{lem:conv}]
		We recall the definition of the convexification:
		\begin{equation*}
			\widehat{\Phi}(\aa) = \sup\{h(\aa): h \text{ is an affine function and } h \leq \Phi\}.
		\end{equation*}
		Then, $\widehat{\Phi}(\aa) = \Phi(\aa)$ if we can find an affine function below $\Phi$ which is equal to $\Phi$ at point $\aa$. In the following, we will prove that if there exists such an $\aa\in B_{\delta}(\ww)$ for $\delta$ small enough, then the linear approximation of $\Phi$ at $\aa$, which we denote by $L_{\aa}\Phi$, is such an affine function.
		
		It is obvious that $\Phi(\aa) = L_{\aa}\Phi(\aa)$. It remains to show that 
		\begin{equation}\label{main1}
			\Phi(\cc) \geq L_{\aa}\Phi(\cc) \quad \text{ for all }\quad \cc\in \mathbb R^N_{>0}.
		\end{equation}
		In order to do so, recall from \eqref{tt1} that $\Phi(\cc) = F(\cc) + G(\cc)$. We will treat $F$ and $G$ separately.
		 
		First, by using the notations $\rho = |\cc - \aa|$, $\delta = |\aa - \ww|$, and direct computation of the Taylor expansion, we have thanks to the convexity of $F$:
		\begin{equation*}
			F(\cc) - L_{\aa}F(\cc) \geq \frac{\nu \rho^2}{1+\rho},
		\end{equation*}
		for some constant $\nu>0$.
		
		Next, thanks to the smoothness of $G$, and $G(\ww) = DG(\ww) = 0$ and $D^2G(\ww) \geq 0$, we get for $\aa \in B_{\varepsilon_*}(\ww)$ 
		(with $\varepsilon_*$ small enough)
		\begin{equation*}
			|G(\aa)| \leq M\delta^2, \quad |DG(\aa)| \leq M\delta, \quad D^2G(\aa) \geq -M\delta, \quad |D^3G(\aa)| \leq M.
		\end{equation*}
		By using Taylor expansion again, we see that locally:
		\begin{equation}\label{rho}
			G(\cc) - L_{\aa}G(\cc) \geq -M\delta\rho^2 - M\rho^3.
		\end{equation}
When $\cc$ is far away from $\aa$ and $\rho$ is not small, then by using the fact that $\aa$ is close to $\ww$ and $G(\ww)$ is zero only at point $\ww$, we obtain that $G(\cc) - L_{\aa}G(\cc)$ is bounded below. Thus, \eqref{rho} holds also in this case for some suitable $M>0$. On the other hand, we can also estimate
		\begin{equation}\label{rho1}
			G(\cc) - L_{\aa}G(\cc) \geq -M\delta^2 - M\delta\rho.
		\end{equation}
		Hence, we obtain from \eqref{rho} and \eqref{rho1} the lower bound
		\begin{equation*}
			G(\cc) - L_{\aa}G(\cc) \geq -M(\delta+\rho)\min\{\delta, \rho^2\} \geq -\frac{4M\sqrt{\delta} \rho^2}{1+\rho}.
		\end{equation*}
		Now, by choosing $\delta = \min\{\varepsilon_*, \nu^2/(4M)^2\}$, we have for all $\aa \in B_{\delta}(\ww)$,
		\begin{equation*}
			\Phi(\cc) - L_{\aa}\Phi(\cc) \geq \frac{(\nu - 4M\sqrt{\delta})\rho^2}{1+\rho} \geq 0.
		\end{equation*}
		This proves \eqref{main1} and thus completes the proof of the Lemma.
	\end{proof}

\vskip 0.5cm
\noindent{\bf Acknowledgements.} This work was carried out during the visit of the third author to ENS Cachan and the visit of the second and the third authors to Universit\'{e} Paris Diderot. The universities' hospitality is gratefully acknowledged. The third author is supported by International Research Training Group IGDK 1754
and by a public grant as part of the Investissement d'avenir project, reference ANR-11-LABX-0056-LMH, LabEx LMH. This work has partially been supported by NAWI Graz.  The research leading to this paper was also funded by
the French ``ANR blanche'' project Kibord: ANR-13-BS01-0004, and by Universit\'e Sorbonne Paris Cit\'e, in the framework of the ``Investissements d'Avenir'', convention ANR-11-IDEX-0005.

\end{document}